\documentclass[11pt,letterpaper]{amsart}

\usepackage[T1]{fontenc}
\usepackage{lmodern}
\usepackage{amsmath,amssymb,amsthm,mathtools}
\usepackage{aliascnt}
\usepackage{microtype}
\usepackage{enumitem}
\usepackage{booktabs,tabularx,array,float}
\usepackage[hidelinks]{hyperref}

\newcolumntype{Y}{>{\raggedright\arraybackslash}X}
\newcolumntype{P}[1]{>{\raggedright\arraybackslash}p{#1}}

\allowdisplaybreaks

\newtheorem{theorem}{Theorem}[section]
\newaliascnt{proposition}{theorem}
\newtheorem{proposition}[proposition]{Proposition}
\aliascntresetthe{proposition}
\newaliascnt{lemma}{theorem}
\newtheorem{lemma}[lemma]{Lemma}
\aliascntresetthe{lemma}
\newaliascnt{corollary}{theorem}
\newtheorem{corollary}[corollary]{Corollary}
\aliascntresetthe{corollary}
\theoremstyle{definition}
\newaliascnt{definition}{theorem}
\newtheorem{definition}[definition]{Definition}
\aliascntresetthe{definition}
\newaliascnt{example}{theorem}
\newtheorem{example}[example]{Example}
\aliascntresetthe{example}
\theoremstyle{remark}
\newaliascnt{remark}{theorem}
\newtheorem{remark}[remark]{Remark}
\aliascntresetthe{remark}

\usepackage[nameinlink,noabbrev]{cleveref}

\crefname{theorem}{theorem}{theorems}
\crefname{proposition}{proposition}{propositions}
\crefname{lemma}{lemma}{lemmas}
\crefname{corollary}{corollary}{corollaries}
\crefname{definition}{definition}{definitions}
\crefname{example}{example}{examples}

\newcommand{\1}{\mathbf 1}
\newcommand{\eps}{\varepsilon}

\newcommand{\ip}[2]{\langle #1,#2\rangle_Q}

\title{Inverse Theorems of Point--Sphere Incidences over Finite Fields}
\author{Shalender Singh}
\author{Vishnu Priya Singh Parmar}
\date{}

\hypersetup{
  pdftitle={Inverse Theorems of Point--Sphere Incidences over Finite Fields},
  pdfauthor={Shalender Singh and Vishnu Priya Singh Parmar},
  pdfsubject={Finite-field incidence geometry},
  pdfkeywords={point-sphere incidences, finite fields, Freiman-type inverse theorem, lifted-linear sphere systems, hypergraph universality, metric entropy, rational normal curves, twisted cubic, incidence designs, fixed-radius spheres, radical hyperplanes, block decompositions, pinned distances, pinned dot products}
}

\begin{document}

\begin{abstract}
Let $Q$ be a nondegenerate quadratic form on $\mathbb F_q^d$, with $q$ odd and $d\ge2$.  We prove a two-scale inverse theory for point--sphere incidences.

At the maximal-incidence endpoint, put $\mu=1-I(P,\mathcal S)/(|P||\mathcal S|)$.  If $\mu<(d+1)/(d+2)^2$, then deleting at most a proportion $\sqrt{(d+1)\mu}$ from each side places the retained configuration in one explicit lifted-linear model: the sphere parameters
\[
 \Phi(S_Q(c,r))=(2c,r-Q(c))
\]
lie in an affine space $L$ of dimension at most $d$, and the retained points lie in its common affine-quadric base.  The natural deletion distance satisfies
\[
 \mu\le \operatorname{dist}_{\mathrm{lin}}(P,\mathcal S)
 \le2\sqrt{(d+1)\mu},
\]
and the square-root exponent is optimal.  Under codimension-two nonconcentration, and provided at least two spheres survive, the model reduces to one coaxal pencil.

This classification has a sharp lower-scale boundary.  A polynomial-root chart on one lifted rational normal curve realizes every admissible $d$-uniform hypergraph; consequently any list of individual point-set templates approximating all resulting constant-$K$ systems within $o(q^d)$ edits has size $\exp(\Omega_d(q^d))$.  Exact fixed-nonzero-radius centered direct sums occur even when $|\mathcal S|\ge4q/K^2$, and a prescribed-radius twisted-cubic system in dimension three has a complete rich-pair graph but neither a heavy pair section nor a quantitatively nontrivial centered block decomposition.

At the fixed-radius deviation scale, what remains universally true is sharp adaptive extraction.  If
\[
 \sigma(P,\mathcal S)\ge Kq^{(d-1)/2}\sqrt{|P||\mathcal S|},
 \qquad K^2q^{d-1}|\mathcal S|\ge4|P|,
\]
then at least $K^2|\mathcal S|/4$ ordered pairs have a common $P$-section of size at least $K^2q^{d-1}/(4|\mathcal S|)$; for $d\ge3$ the witness count improves to $K^2q|\mathcal S|/8$.  We also prove mixed-radius profile stability for both signs, a log-free refinement, a cap-free positive-surplus decomposition, and applications to affine hyperplanes, pinned distances, and pinned dot products.
\end{abstract}

\maketitle

\section{Introduction}

Point--sphere incidences over finite fields connect finite geometry, pseudorandom incidence graphs, additive combinatorics, and distance problems.  Let $Q$ be a nondegenerate quadratic form on $\mathbb F_q^d$, where $q$ is odd, let $P\subseteq\mathbb F_q^d$, and let $\mathcal S$ be a family of distinct $Q$-spheres.  We write
\[
 \sigma(P,\mathcal S)=I(P,\mathcal S)-\frac{|P||\mathcal S|}{q}
\]
for the centered incidence count.  The central inverse question is not only whether large incidence forces a structured subset, but whether every near-extremizer is close to an explicit family of models.

The results have three layers, ordered by structural strength.  At the maximal-incidence endpoint we prove a global deletion-stability classification by explicit lifted-linear models.  We next establish a sharp boundary for that classification: at the smaller fixed-radius deviation scale, universal and exact obstruction families preclude a subexponential list of individual edit templates.  Within that broader lower-scale class we prove the strongest universal conclusion available from coincidence energy, namely scale-adaptive extraction and positive-surplus decomposition.

\subsection{Dense-endpoint classification}

The largest possible incidence count is $I(P,\mathcal S)=|P||\mathcal S|$.  Put
\[
 \mu(P,\mathcal S)=1-\frac{I(P,\mathcal S)}{|P||\mathcal S|}.
\]
Thus $\mu$ is the density of missing incidences.  Our first main theorem classifies every system with sufficiently small $\mu$.

Use the lift
\[
 \Phi(S_Q(c,r))=(2c,r-Q(c))\in\mathbb F_q^{d+1}.
\]
For an affine subspace $L\subseteq\mathbb F_q^{d+1}$, let $\mathcal S(L)=\Phi^{-1}(L)$ and let $B(L)$ be the set of points lying on every sphere of $\mathcal S(L)$.  The exact models are completely explicit.  If $L$ has rank $r$, then for an $r$-dimensional subspace $W\le\mathbb F_q^d$, a linear functional $\beta$ on $W$, and suitable $c_0,z,r_0$, one has
\begin{align*}
 \mathcal S(L)
 &=\left\{S_Q\left(c_0+\frac w2,
 r_0+\beta(w)+\langle c_0,w\rangle_Q+\frac14Q(w)\right):w\in W\right\},\\
 B(L)&=S_Q(c_0,r_0)\cap(z+W^{\perp_Q}).
\end{align*}
Rank one is a coaxal pencil; higher ranks are explicit higher-dimensional linear sphere systems with a common affine-quadric base.

\Cref{thm:dense-freiman} proves that if
\[
 0\le\mu<\frac{d+1}{(d+2)^2},
 \qquad
 \delta=\sqrt{(d+1)\mu},
\]
then there is such an $L$, with $\dim L\le d$, for which at least $(1-\delta)|P|$ points lie in $B(L)$ and at least $(1-\delta)|\mathcal S|$ spheres lie in $\mathcal S(L)$.  Hence, after controlled deletion from both sides, the configuration lies in one exact lifted-linear ambient model; for fixed $d$ there are only $q^{O_d(1)}$ possible ambient spaces.

The classification is two-sided in a natural edit metric.  If $\operatorname{dist}_{\mathrm{lin}}(P,\mathcal S)$ denotes the least total deleted proportion needed to place both sides in one lifted-linear model, then
\[
 \mu(P,\mathcal S)
 \le \operatorname{dist}_{\mathrm{lin}}(P,\mathcal S)
 \le 2\sqrt{(d+1)\mu(P,\mathcal S)}.
\]
A split four-dimensional construction shows that the exponent $1/2$ is optimal.  Moreover, if no affine codimension-two flat contains almost all of $P$, the classifying rank is at most one; provided two spheres remain, the model is one coaxal pencil.  These statements form the paper's Freiman-type inverse theorem, in the ambient-model and edit-distance sense familiar from classical inverse theory \cite{GR,GT}.

\subsection{The classification boundary at the fixed scale}

A classification comparable to the dense-endpoint theorem is impossible in the narrower sense of approximation by a subexponential list of individual point-set templates.  We construct $q$ nonzero-radius spheres whose lifted parameters form a rational normal curve and a polynomial-root chart with the following universality property: for every admissible $d$-uniform hypergraph $\mathcal G$ on $\mathbb F_q$, a point set $P_{\mathcal G}$ has incidence graph
\[
 p_A\in S_t\quad\Longleftrightarrow\quad t\in A.
\]
Dense choices have constant fixed-scale parameter and lie in the moderate range.  A Hamming-ball argument shows that any list of individual templates approximating every such $P_{\mathcal G}$ within $o(q^d)$ point edits must contain $\exp(\Omega_d(q^d))$ members.  This is an edit-metric obstruction to an individual-template classification; it does not rule out container statements whose model classes themselves permit arbitrary dense subhypergraphs.

The full root design exhibits a particularly rigid diffuse regime: every sphere pair is rich, no pair section carries more than $O_d(q^{-2})$ of the surplus, no three lifted parameters are coaxal, and no quantitatively nontrivial centered block decomposition exists.  A second obstruction is an exact centered direct sum.  For every nondegenerate $Q$ in $d\ge3$, at one fixed nonzero radius, we construct matched blocks satisfying
\[
 \sigma(A_b,\mathcal S_z)=0\quad(b\ne z),
 \qquad
 \sigma(P,\mathcal S)=\sum_z\sigma(A_z,\mathcal S_z).
\]
In $d\ge4$ these examples occur in the family-only range $|\mathcal S|\ge4q/K^2$ while no hyperplane or pair section is macroscopic.  In dimension three a prescribed-radius twisted-cubic chart realizes the diffuse root design at the sharp point-sensitive boundary.  Thus the lower-scale theory has at least three regimes: coaxal concentration, centered direct sums, and diffuse algebraic evaluation designs.

\subsection{Scale-adaptive inverse theory at the fixed-radius benchmark}

For one fixed nonzero radius, the natural Fourier deviation scale is
\[
 q^{(d-1)/2}\sqrt{|P||\mathcal S|},
\]
whereas arbitrary mixed radii admit the larger scale $q^{d/2}\sqrt{|P||\mathcal S|}$; see \cite{IR,CILLR,NPV,KP}.  Totally singular pencils show that the factor $q^{1/2}$ between these scales is genuine.

At the fixed-radius benchmark our inverse theorem is scale-adaptive.  Write $n=|P|$ and $s=|\mathcal S|$, and suppose
\begin{equation}\label{eq:intro-fixed}
 \sigma(P,\mathcal S)\ge Kq^{(d-1)/2}\sqrt{ns},
 \qquad
 K^2q^{d-1}s\ge4n.
\end{equation}
Then at least $K^2s/4$ ordered pairs satisfy
\begin{equation}\label{eq:intro-adaptive}
 |P\cap S\cap S'|\ge \frac{K^2q^{d-1}}{4s}.
\end{equation}
No upper bound on $s$ is needed.  Every witnessing pair determines a canonical $q$-member coaxal pencil with the same common section, and the section lies in its radical hyperplane.  If $d\ge3$, the pair-codegree bound $|S\cap S'|\le2q^{d-2}$ improves the witness count to $K^2qs/8$.  Under the moderate cap $s\le C_1Kq^{(d-1)/2}$, \eqref{eq:intro-adaptive} becomes
\[
 |P\cap S\cap S'|\ge \frac{Kq^{(d-1)/2}}{4C_1}.
\]
A single high-degree truncation gives a log-free incidence-rich refinement.  Iteration in the positive-excess core gives a cap-free decomposition in which all but a prescribed fraction of the net surplus is carried by disjoint positive-surplus pieces supported on sphere-pair sections.  For completeness, exact sphere-complement duality also gives an omitted-family corollary.  Its section-size threshold is below $1$ throughout the balanced regime $|\mathcal S|\le|\Omega|/2$, and it becomes structurally nontrivial only when the omitted family is an $O(q^{-1})$ fraction of the full sphere system; we therefore do not count it among the main lower-scale inverse results.

\subsection{Relation to previous work and applications}

Koh, Lund, Xu, and Yoo use pairwise sphere intersections in the forward direction, deriving incidence estimates from codegree control \cite{KLXY}.  Senger and Tran obtain forward estimates under higher-order Salem hypotheses \cite{ST}.  Our use of pair intersections is inverse: observed discrepancy is converted into actual rich sections and then into a decomposition or an obstruction to classification.  The following table compares direction, hypotheses, and structural output; it does not assert numerical domination between estimates.

\begin{table}[ht]
\centering
\small
\caption{Comparison with the closest recent forward results.  The exponent $s_0$ in \cite{ST} is unrelated to $|\mathcal S|$.}
\label{tab:comparison}
\begin{tabularx}{\textwidth}{@{}P{0.17\textwidth}Y Y@{}}
\toprule
Work & Hypotheses and quantitative output & Inverse or structural content \\
\midrule
Koh--Lund--Xu--Yoo \cite{KLXY}
& General point sets in sphere regimes where sums of pair intersections can be controlled; forward incidence estimates follow from those codegree bounds.
& Pair intersections are input to an upper bound for $I(P,\mathcal S)$; no inverse classification is asserted. \\
\addlinespace
Senger--Tran \cite{ST}
& $P$ is a $(4,s_0)$-Salem set with $s_0\in(1/4,1/2]$ and $|P|\ll q^{d/(4s_0)}$; then $|\sigma(P,\mathcal S)|\ll q^{d/4}|P|^{1-s_0}|\mathcal S|^{3/4}$.
& Pseudorandomness and the stated size range for $P$ are assumed; no inverse conclusion is sought. \\
\addlinespace
Present work
& Arbitrary $P$ and nondegenerate $Q$; dense-endpoint near-maximal incidence or fixed-scale discrepancy with diagonal dominance.
& Global deletion classification at the dense endpoint; adaptive section extraction, surplus decomposition, and sharp obstruction families at the lower scale. \\
\bottomrule
\end{tabularx}
\end{table}

The exact centered design of all mixed-radius spheres gives a complement-sensitive mixing inequality and a quantitative two-level classification of equality and near-equality profiles for both signs.  The same coincidence principle applies to affine hyperplanes, yielding rich codimension-two intersections.  Exact regularity in pin-generated systems gives deficiency-sensitive pin-count bounds and inverse theorems for deficient pinned distances and pinned dot products.

\subsection{Scope and organization}

The dense-endpoint theorem is a global classification after deletion.  The fixed-scale results are sharp extraction and decomposition theorems, not a classification of every constant-$K$ system.  The rational-normal-curve theorem rules out only subexponential lists of individual edit templates, while the exact block and diffuse constructions identify distinct obstruction mechanisms.  The auxiliary complement-duality statement is structurally informative only in its narrow omitted-family range.  An intrinsic deficit theory inside one prescribed fixed-radius class remains tied to near-extremal Kloosterman and Sali\'e directions.

\Cref{sec:geometry} develops exact designs, pair geometry, the dense-endpoint classification, and complete-pencil examples.  \Cref{sec:inverse} proves the adaptive extraction theorem and its refinements.  \Cref{sec:atomicity} constructs exact fixed-radius blocks and diffuse rational-normal-curve systems.  \Cref{sec:applications} treats affine hyperplanes and pinned configurations.  \Cref{sec:discussion} summarizes the classification boundary, and \cref{app:verification} records the exact-arithmetic verification suite.

\section{Sphere geometry and exact design identities}\label{sec:geometry}

\subsection{Quadratic spheres}

Throughout, $q$ is an odd prime power, $d\ge2$, and $Q:\mathbb F_q^d\to\mathbb F_q$ is a nondegenerate quadratic form.  Its associated symmetric bilinear form is
\[
 \ip{x}{y}=\frac{Q(x+y)-Q(x)-Q(y)}2.
\]
For $c\in\mathbb F_q^d$ and $r\in\mathbb F_q$, define the $Q$-sphere
\[
 S_Q(c,r)=\{x\in\mathbb F_q^d:Q(x-c)=r\}.
\]
Spheres are parametrized by $(c,r)$, and distinct spheres means distinct parameter pairs.  We suppress the subscript $Q$ when no confusion can arise.

For $P\subseteq\mathbb F_q^d$ and a family $\mathcal S$ of distinct spheres, let
\[
 I(P,\mathcal S)=|\{(p,S)\in P\times\mathcal S:p\in S\}|,
 \qquad
 \sigma(P,\mathcal S)=I(P,\mathcal S)-\frac{|P||\mathcal S|}{q}.
\]
For $p\in P$, write $d_{\mathcal S}(p)=|\{S\in\mathcal S:p\in S\}|$.

\subsection{The full sphere design}

Let $\Omega$ be the family of all $q^{d+1}$ parametrized $Q$-spheres, including radius zero.

\begin{proposition}[Centered sphere design]\label{prop:sphere-design}
Let $A$ be the $q^d\times q^{d+1}$ point--sphere incidence matrix, let $J$ be the all-ones $q^d\times q^d$ matrix, and let $J_0$ be the all-ones matrix of the same size as $A$.
\begin{enumerate}[label=\textup{(\arabic*)}]
\item Every point lies on exactly $q^d$ spheres and every pair of distinct points lies on exactly $q^{d-1}$ common spheres.  Hence
\[
 AA^{\mathsf T}=q^{d-1}J+(q^d-q^{d-1})I.
\]
\item If $B=A-q^{-1}J_0$, then
\[
 BB^{\mathsf T}=(q^d-q^{d-1})I,
 \qquad B\1_{\Omega}=0.
\]
Equivalently, for every $P\subseteq\mathbb F_q^d$,
\[
 \sum_{T\in\Omega}
 \left(|P\cap T|-\frac{|P|}{q}\right)^2
 =(q^d-q^{d-1})|P|.
\]
\item For every $P\subseteq\mathbb F_q^d$ and $\mathcal S\subseteq\Omega$,
\begin{equation}\label{eq:complement-sphere}
 |\sigma(P,\mathcal S)|
 \le \sqrt{q^d-q^{d-1}}
 \sqrt{|P||\mathcal S|
 \left(1-\frac{|\mathcal S|}{q^{d+1}}\right)}.
\end{equation}
Equality holds when $P$ is a singleton and $\mathcal S$ is the family of all spheres containing its point.
\end{enumerate}
\end{proposition}

\begin{proof}
Fix $p\in\mathbb F_q^d$.  Each center $c$ determines the unique radius $Q(p-c)$, so $p$ lies on $q^d$ spheres.  If $p\ne p'$, then both points lie on $S(c,r)$ precisely when
\[
 Q(p-c)=Q(p'-c),
\]
equivalently
\[
 2\ip{p'-p}{c}=Q(p')-Q(p).
\]
Nondegeneracy makes this a nontrivial affine-linear equation in $c$, with exactly $q^{d-1}$ solutions.  This proves the first assertion.

Each row of $A$ has sum $q^d$ and $|\Omega|=q^{d+1}$.  Expanding $BB^{\mathsf T}$ cancels the $q^{d-1}J$ term and gives the second assertion, including $B\1_\Omega=0$.

Put $v=B^{\mathsf T}\1_P$.  Then $v\perp\1_\Omega$ and
\[
 \|v\|_2^2=(q^d-q^{d-1})|P|.
\]
Consequently
\[
 \sigma(P,\mathcal S)
 =\left\langle v,\1_{\mathcal S}
 -\frac{|\mathcal S|}{|\Omega|}\1_\Omega\right\rangle,
\]
and Cauchy--Schwarz gives \eqref{eq:complement-sphere}.  In the point-star example the two centered vectors coincide, proving equality.
\end{proof}

\begin{theorem}[Sharp mixed-radius profile stability and equality classification]\label{thm:mixed-profile}
Let $P\subseteq\mathbb F_q^d$ be nonempty, put $n=|P|$, let $\varnothing\ne\mathcal S\subsetneq\Omega$, put $s=|\mathcal S|$, and write
\[
 M=|\Omega|=q^{d+1},
 \qquad
 \alpha=q^d-q^{d-1},
 \qquad
 \Delta=\sqrt{\alpha n s(1-s/M)}.
\]
Suppose $0\le\delta<1$ and
\begin{equation}\label{eq:mixed-near-equality}
 |\sigma(P,\mathcal S)|\ge(1-\delta)\Delta.
\end{equation}
Let $\epsilon=\operatorname{sgn}\sigma(P,\mathcal S)$.  Then the sphere-intersection profile satisfies
\begin{equation}\label{eq:mixed-sphere-profile}
 \sum_{T\in\Omega}
 \left(
 \frac{|P\cap T|-n/q}{\sqrt{\alpha n}}
 -\epsilon\,
 \frac{\1_{\mathcal S}(T)-s/M}{\sqrt{s(1-s/M)}}
 \right)^2
 \le2\delta,
\end{equation}
and the point-degree profile satisfies
\begin{equation}\label{eq:mixed-point-profile}
 \sum_{x\in\mathbb F_q^d}
 \left(
 \frac{\1_P(x)}{\sqrt n}
 -\epsilon\,
 \frac{d_{\mathcal S}(x)-s/q}
 {\sqrt{\alpha s(1-s/M)}}
 \right)^2
 \le2\delta.
\end{equation}
In particular, equality in \eqref{eq:complement-sphere} holds if and only if the following equivalent two-level identities hold:
\begin{align}
 |P\cap T|-\frac nq
 &=\epsilon\sqrt{\frac{\alpha n}{s(1-s/M)}}
 \left(\1_{\mathcal S}(T)-\frac sM\right)
 &&(T\in\Omega),\label{eq:mixed-equality-spheres}\\
 d_{\mathcal S}(x)-\frac sq
 &=\epsilon\sqrt{\frac{\alpha s(1-s/M)}n}\,\1_P(x)
 &&(x\in\mathbb F_q^d).\label{eq:mixed-equality-points}
\end{align}
Thus both positive and negative extremizers, and all near-extremizers at the sharp mixed-radius bound, are classified at the level of their complete sphere-intersection and point-degree profiles.
\end{theorem}

\begin{proof}
Let $B$ be the centered incidence matrix from \cref{prop:sphere-design}, put
\[
 u=B^{\mathsf T}\1_P,
 \qquad
 w=\1_{\mathcal S}-\frac sM\1_\Omega.
\]
Then $u\perp\1_\Omega$, $\|u\|_2^2=\alpha n$, $\|w\|_2^2=s(1-s/M)$, and
\[
 \sigma(P,\mathcal S)=\langle u,w\rangle.
\]
After normalizing, \eqref{eq:mixed-near-equality} gives
\[
 \left\|
 \frac u{\sqrt{\alpha n}}
 -\epsilon\frac w{\sqrt{s(1-s/M)}}
 \right\|_2^2
 =2-2\frac{|\sigma(P,\mathcal S)|}{\Delta}
 \le2\delta,
\]
which is \eqref{eq:mixed-sphere-profile}.

Since $BB^{\mathsf T}=\alpha I$, every singular value of $B$ is at most $\sqrt\alpha$.  Apply $B$ to the preceding normalized-vector difference and divide by $\sqrt\alpha$.  The identities
\[
 Bu=BB^{\mathsf T}\1_P=\alpha\1_P,
 \qquad
 Bw=B\1_{\mathcal S}
 =\bigl(d_{\mathcal S}(x)-s/q\bigr)_{x\in\mathbb F_q^d}
\]
give \eqref{eq:mixed-point-profile}.  When $\delta=0$, the normalized vectors are equal up to the sign $\epsilon$, yielding \eqref{eq:mixed-equality-spheres} and \eqref{eq:mixed-equality-points}.  Conversely, \eqref{eq:mixed-equality-spheres} directly gives equality in Cauchy--Schwarz.  If \eqref{eq:mixed-equality-points} holds, then $\|Bw\|_2^2=\alpha\|w\|_2^2$.  Since $BB^{\mathsf T}=\alpha I$, the operator $\alpha^{-1}B^{\mathsf T}B$ is the orthogonal projector onto $\operatorname{im}B^{\mathsf T}$; hence $w\in\operatorname{im}B^{\mathsf T}$ and applying $\alpha^{-1}B^{\mathsf T}$ to \eqref{eq:mixed-equality-points} recovers \eqref{eq:mixed-equality-spheres}.  Thus either identity is equivalent to equality in \eqref{eq:complement-sphere}.
\end{proof}

\begin{remark}\label{rem:classification-scope}
\Cref{thm:mixed-profile} is a genuine stability classification at the exact mixed-radius design norm.  It is intentionally a profile theorem: it does not assert that every two-level profile comes from a unique geometric model.  A geometric Freiman-type classification is proved at the maximal-incidence endpoint in \cref{thm:dense-freiman,cor:model-distance}; at the smaller fixed-radius benchmark, \cref{thm:hypergraph-universality,cor:no-freiman-list} show that no subexponential list of individual edit templates is possible.
\end{remark}

Dropping the final factor in \eqref{eq:complement-sphere} gives the standard mixed-radius bound
\begin{equation}\label{eq:mixed-bound}
 |\sigma(P,\mathcal S)|
 \le \sqrt{1-q^{-1}}\,q^{d/2}\sqrt{|P||\mathcal S|}.
\end{equation}

\begin{proposition}[Fixed nonzero radius]\label{prop:fixed-radius}
Let $r\in\mathbb F_q^\times$, and let $\mathcal S_r$ be a family of distinct spheres of radius $r$.  Then
\begin{equation}\label{eq:fixed-radius-bound}
 |\sigma(P,\mathcal S_r)|
 \le 4q^{(d-1)/2}\sqrt{|P||\mathcal S_r|}.
\end{equation}
The constant $4$ is inessential; in particular the deviation scale is $q^{(d-1)/2}$ for every nondegenerate $Q$.
\end{proposition}

\begin{proof}
Fix a nontrivial additive character $\chi$ of $\mathbb F_q$ and identify the dual space with $\mathbb F_q^d$ through $\langle\cdot,\cdot\rangle_Q$.  For $h:\mathbb F_q^d\to\mathbb C$, use the unnormalized transform
\[
 \widehat h(\xi)=\sum_x h(x)\chi(-\ip{\xi}{x}).
\]
Let $V_r=\{x:Q(x)=r\}$.  Orthogonality and completion of the square give, for $\xi\ne0$,
\begin{align*}
 \widehat{\1_{V_r}}(\xi)
 &=q^{-1}\sum_{a\in\mathbb F_q^\times}\chi(-ar)
   \sum_x\chi\bigl(aQ(x)-\ip{\xi}{x}\bigr)\\
 &=q^{-1}\sum_{a\in\mathbb F_q^\times}\chi\left(-ar-\frac{Q(\xi)}{4a}\right)G_Q(a),
\end{align*}
where $G_Q(a)=\sum_x\chi(aQ(x))$ has modulus $q^{d/2}$ and equals a fixed unit times $\eta(a)^d q^{d/2}$, with $\eta$ the quadratic character.  The remaining sum is a Kloosterman sum when $d$ is even and a Sali\'e sum when $d$ is odd.  Since $r\ne0$, the Weil bound yields
\begin{equation}\label{eq:fixed-fourier}
 \max_{\xi\ne0}|\widehat{\1_{V_r}}(\xi)|
 \le2q^{(d-1)/2}.
\end{equation}
The same calculation at $\xi=0$ gives
\begin{equation}\label{eq:level-size}
 \bigl||V_r|-q^{d-1}\bigr|\le2q^{(d-1)/2}.
\end{equation}
See, for example, \cite[Chapters~5--6]{LN} for the Gauss-sum evaluation and the Weil bounds.

Let $C$ be the set of centers of the spheres in $\mathcal S_r$.  Fourier inversion and Parseval give
\[
 \left|I(P,\mathcal S_r)-\frac{|P||C||V_r|}{q^d}\right|
 \le2q^{(d-1)/2}\sqrt{|P||C|}.
\]
By \eqref{eq:level-size}, the difference between the displayed main term and $|P||C|/q$ is at most
\[
 2q^{-(d+1)/2}|P||C|
 \le2q^{(d-1)/2}\sqrt{|P||C|},
\]
because $|P|,|C|\le q^d$.  Since $|C|=|\mathcal S_r|$, the result follows.
\end{proof}

Thus $q^{(d-1)/2}$ is the natural deviation scale per nonzero radius, whereas \eqref{eq:mixed-bound} shows that unrestricted radii have scale $q^{d/2}$.  The normalization in \eqref{eq:intro-fixed} applies the fixed-radius benchmark as a structural threshold to arbitrary sphere families.

\subsection{Radical hyperplanes and pair codegrees}

\begin{lemma}[Radical hyperplane and pair bound]\label{lem:radical}
Let $S=S(c,r)$ and $S'=S(c',r')$ be distinct spheres.  If $c=c'$, then $S\cap S'=\varnothing$.  If $c\ne c'$, then
\[
 S\cap S'\subseteq H(S,S'),
\]
where
\begin{equation}\label{eq:radical}
 H(S,S')=
 \left\{x\in\mathbb F_q^d:
 2\ip{c'-c}{x}=(r-r')-Q(c)+Q(c')\right\}.
\end{equation}
Thus $|S\cap S'|\le q^{d-1}$.  If $d\ge3$, then
\begin{equation}\label{eq:pair-bound}
 |S\cap S'|\le2q^{d-2}.
\end{equation}
\end{lemma}

\begin{proof}
Subtracting $Q(x-c')=r'$ from $Q(x-c)=r$ gives \eqref{eq:radical}.  If $c=c'$ and $r\ne r'$, the two level sets are disjoint.  If $c\ne c'$, nondegeneracy makes \eqref{eq:radical} a genuine affine hyperplane.

For $d\ge3$, restrict $Q(x-c)-r$ to $H(S,S')$, regarded as an affine space of dimension $d-1$.  Its homogeneous quadratic part is the restriction of $Q$ to the direction subspace of $H(S,S')$.  This restriction is not identically zero: otherwise that $(d-1)$-dimensional subspace would be totally singular, exceeding the Witt-index bound $\lfloor d/2\rfloor$.  The restricted polynomial is therefore nonzero of degree at most two, and the Schwartz--Zippel bound gives \eqref{eq:pair-bound}.
\end{proof}

The $d=2$ exception is genuine: in a split plane, two radius-zero spheres can have an entire isotropic line in common.

\begin{lemma}[Canonical coaxal completion]\label{lem:coaxal}
For a sphere $S=S_Q(c,r)$ put
\[
 \Phi(S)=(2c,r-Q(c))\in\mathbb F_q^{d+1}.
\]
If $S_0,S_1$ are distinct spheres, then the $q$ points
\[
 \Phi(S_t)=(1-t)\Phi(S_0)+t\Phi(S_1),
 \qquad t\in\mathbb F_q,
\]
parametrize $q$ distinct spheres.  For every $t\ne u$,
\[
 S_t\cap S_u=S_0\cap S_1.
\]
Thus a sphere pair determines a canonical $q$-member coaxal pencil whose base locus is exactly its common section.
\end{lemma}

\begin{proof}
The equation of $S_Q(c,r)$ can be written
\[
 Q(x)-2\ip{c}{x}=r-Q(c).
\]
Taking affine combinations of the two endpoint equations gives the sphere with parameter $\Phi(S_t)$.  Since $\Phi(S_0)\ne\Phi(S_1)$, the $q$ parameters are distinct.  A point satisfying the two equations indexed by $t\ne u$ satisfies both endpoint equations, and the converse is immediate.
\end{proof}

\subsection{Freiman-type classification at the dense endpoint}\label{subsec:dense-freiman}

The fixed-scale results of \cref{sec:inverse} begin near the random baseline and extract local common sections.  At the opposite, maximal-incidence endpoint one can classify the entire configuration, up to deletion, by a finite-rank linear system in lifted parameter space.  This is the paper's Freiman-type inverse theorem.

For $y=(u,a)\in\mathbb F_q^d\times\mathbb F_q$ and $x\in\mathbb F_q^d$, put
\begin{equation}\label{eq:lambda-lift}
 \Lambda_x(y)=Q(x)-\ip{x}{u}-a.
\end{equation}
Thus, for every sphere $S$,
\begin{equation}\label{eq:lift-incidence-functional}
 x\in S\quad\Longleftrightarrow\quad \Lambda_x(\Phi(S))=0.
\end{equation}
If $L\subseteq\mathbb F_q^{d+1}$ is an affine subspace, define its sphere system and base locus by
\begin{equation}\label{eq:lift-linear-model}
 \mathcal S(L)=\Phi^{-1}(L),
 \qquad
 B(L)=\{x\in\mathbb F_q^d:\Lambda_x(y)=0\text{ for every }y\in L\}.
\end{equation}
Every point of $B(L)$ lies on every sphere of $\mathcal S(L)$.

\begin{proposition}[Exact lifted-linear models and normal form]\label{prop:linear-model-normal-form}
Let $P\ne\varnothing$ and let $\mathcal S$ be a nonempty family of distinct $Q$-spheres.  Then
\[
 I(P,\mathcal S)=|P||\mathcal S|
\]
if and only if, with $L=\operatorname{aff}\Phi(\mathcal S)$,
\[
 P\subseteq B(L),\qquad \mathcal S\subseteq\mathcal S(L).
\]
Whenever $B(L)\ne\varnothing$, one has $\dim L\le d$, and the model admits the following explicit normal form.  Write $L=y_0+U$, where
\[
 y_0=\Phi(S_Q(c_0,r_0)).
\]
Then there are an $r$-dimensional subspace $W\le\mathbb F_q^d$, a linear functional $\beta:W\to\mathbb F_q$, and $z\in\mathbb F_q^d$, where $r=\dim L$, such that
\begin{align}
 \mathcal S(L)
 &=\left\{
 S_Q\left(c_0+\frac w2,
 r_0+\beta(w)+\ip{c_0}{w}+\frac14Q(w)\right):w\in W
 \right\},\label{eq:linear-system-normal-form}\\
 B(L)
 &=S_Q(c_0,r_0)\cap\bigl(z+W^{\perp_Q}\bigr),
 \qquad
 \ip{z}{w}=-\beta(w)\quad(w\in W).
 \label{eq:base-normal-form}
\end{align}
Conversely, every such choice has the displayed base locus; when that base is nonempty, any nonempty subsets of it and of the sphere system form a complete point--sphere model.  Rank $r=1$ is precisely a coaxal pencil; higher ranks are explicit higher-dimensional linear sphere systems with a common affine-quadric base.
\end{proposition}

\begin{proof}
If every pair in $P\times\mathcal S$ is incident, then for each $x\in P$ the affine-linear functional $\Lambda_x$ vanishes on $\Phi(\mathcal S)$ and hence on its affine span $L$.  Thus $P\subseteq B(L)$.  The converse follows immediately from \eqref{eq:lift-linear-model}.

Suppose $B(L)\ne\varnothing$ and write a direction vector in $U$ as $(w,b)\in\mathbb F_q^d\times\mathbb F_q$.  The projection $U\to\mathbb F_q^d$ is injective: if $(0,b)\in U$ and $x\in B(L)$, then vanishing of $\Lambda_x$ on $y_0+t(0,b)$ for all $t$ forces $b=0$.  Hence $r=\dim U\le d$, the projected space $W$ has dimension $r$, and
\[
 U=\{(w,\beta(w)):w\in W\}
\]
for a unique linear functional $\beta$.  Nondegeneracy of $Q$ gives $z$ with $\ip{z}{w}=-\beta(w)$ for all $w\in W$.

A point $x$ belongs to $B(L)$ exactly when it satisfies the equation corresponding to $y_0$ and
\[
 \ip{x}{w}=-\beta(w)=\ip{z}{w}
 \qquad(w\in W),
\]
which is \eqref{eq:base-normal-form}.  Finally, the lift $y_0+(w,\beta(w))$ has center $c_0+w/2$ and radius
\[
 r_0-Q(c_0)+\beta(w)+Q\left(c_0+\frac w2\right)
 =r_0+\beta(w)+\ip{c_0}{w}+\frac14Q(w),
\]
proving \eqref{eq:linear-system-normal-form}.  The converse is the same calculation in reverse.
\end{proof}

For nonempty $P$ and $\mathcal S$, define the missing-incidence density
\begin{equation}\label{eq:missing-density}
 \mu(P,\mathcal S)=1-\frac{I(P,\mathcal S)}{|P||\mathcal S|}.
\end{equation}
Thus $\mu=0$ is the exact endpoint, while
\[
 \sigma(P,\mathcal S)=(1-\mu-q^{-1})|P||\mathcal S|.
\]

\begin{theorem}[Dense-endpoint classification]\label{thm:dense-freiman}
Let $d\ge2$, let $P\subseteq\mathbb F_q^d$ be nonempty, let $\mathcal S$ be a nonempty family of distinct $Q$-spheres, and put $n=|P|$, $s=|\mathcal S|$, and $\mu=\mu(P,\mathcal S)$.  Let
\[
 0<\theta<\frac1{d+2},\qquad \mu<\theta.
\]
Then there are an affine subspace $L\subseteq\mathbb F_q^{d+1}$ of dimension at most $d$ and subsets $P_0\subseteq P$, $\mathcal S_0\subseteq\mathcal S$ such that
\begin{align}
 P_0&\subseteq B(L),&
 |P_0|&\ge\bigl(1-(d+1)\theta\bigr)n,\label{eq:dense-P0}\\
 \mathcal S_0&\subseteq\mathcal S(L),&
 |\mathcal S_0|&\ge\left(1-\frac\mu\theta\right)s.
 \label{eq:dense-S0}
\end{align}
In particular, every point of $P_0$ lies on every sphere of $\mathcal S_0$, and the retained configuration is one of the explicit rank-$r$ models in \cref{prop:linear-model-normal-form}.

If
\begin{equation}\label{eq:dense-small-mu}
 0<\mu<\frac{d+1}{(d+2)^2}
\end{equation}
and
\[
 \delta=\sqrt{(d+1)\mu},
\]
then one may choose
\begin{equation}\label{eq:dense-symmetric}
 |P_0|\ge(1-\delta)n,
 \qquad
 |\mathcal S_0|\ge(1-\delta)s.
\end{equation}
For $\mu=0$, the same conclusion holds with $P_0=P$, $\mathcal S_0=\mathcal S$, and $\delta=0$.
\end{theorem}

\begin{proof}
Assume first that $\mu>0$.  Call a sphere $S\in\mathcal S$ good when
\[
 |P\setminus S|\le\theta n
\]
and let $\mathcal S_0$ be the good subfamily.  Counting missing incidences by spheres gives
\[
 \theta n\,|\mathcal S\setminus\mathcal S_0|
 \le \mu ns,
\]
which proves \eqref{eq:dense-S0}; in particular $\mathcal S_0\ne\varnothing$ because $\mu<\theta$.

Let $L=\operatorname{aff}\Phi(\mathcal S_0)$ and put $r=\dim L$.  Initially $r\le d+1$.  Choose $r+1$ members $S_0,\ldots,S_r$ whose lifts form an affine basis of $L$, and set
\[
 P_0=P\cap S_0\cap\cdots\cap S_r.
\]
The union bound and goodness give
\[
 |P_0|\ge\bigl(1-(r+1)\theta\bigr)n
 \ge\bigl(1-(d+2)\theta\bigr)n>0.
\]
Fix $x\in P_0$.  The affine-linear functional $\Lambda_x$ vanishes on an affine basis of $L$, hence on all of $L$.  Thus $L$ is contained in the affine hyperplane $\{y:\Lambda_x(y)=0\}$, so $r\le d$.  Repeating the union-bound estimate with this improved dimension gives \eqref{eq:dense-P0}.  The same affine-linearity argument shows $P_0\subseteq B(L)$ and $\mathcal S_0\subseteq\mathcal S(L)$.

Under \eqref{eq:dense-small-mu}, take
\[
 \theta=\sqrt{\frac\mu{d+1}}.
\]
Then $\theta<1/(d+2)$ and $\mu<\theta$, while
\[
 (d+1)\theta=\frac\mu\theta=\sqrt{(d+1)\mu}=\delta.
\]
This proves \eqref{eq:dense-symmetric}.  If $\mu=0$, every incidence is present, and \cref{prop:linear-model-normal-form} applies directly.
\end{proof}

\begin{definition}[Distance from a lifted-linear model]\label{def:linear-distance}
For nonempty $P$ and $\mathcal S$, put
\[
 \operatorname{dist}_{\mathrm{lin}}(P,\mathcal S)
 =\inf_{L:B(L)\ne\varnothing}
 \left(
 \frac{|P\setminus B(L)|}{|P|}
 +\frac{|\mathcal S\setminus\mathcal S(L)|}{|\mathcal S|}
 \right).
\]
\end{definition}

\begin{corollary}[Quantitative model-distance equivalence]\label{cor:model-distance}
Let $\mu=\mu(P,\mathcal S)$.  Then
\begin{equation}\label{eq:model-distance-lower}
 \mu\le\operatorname{dist}_{\mathrm{lin}}(P,\mathcal S).
\end{equation}
If $\mu<(d+1)/(d+2)^2$, then
\begin{equation}\label{eq:model-distance-upper}
 \operatorname{dist}_{\mathrm{lin}}(P,\mathcal S)
 \le2\sqrt{(d+1)\mu}.
\end{equation}
Equivalently, near-maximizers of the sharp endpoint bound
\[
 \sigma(P,\mathcal S)\le(1-q^{-1})|P||\mathcal S|
\]
are, up to an explicit square-root stability loss, exactly the lifted-linear sphere systems of \cref{prop:linear-model-normal-form}.  More explicitly, if
\[
 \sigma(P,\mathcal S)\ge(1-\eta)(1-q^{-1})|P||\mathcal S|,
\]
then $\mu\le\eta(1-q^{-1})$, and whenever this quantity satisfies \eqref{eq:dense-small-mu},
\[
 \operatorname{dist}_{\mathrm{lin}}(P,\mathcal S)
 \le2\sqrt{(d+1)\eta(1-q^{-1})}.
\]
\end{corollary}

\begin{proof}
For an affine $L$ in \cref{def:linear-distance}, set
\[
 a=\frac{|P\setminus B(L)|}{|P|},
 \qquad
 b=\frac{|\mathcal S\setminus\mathcal S(L)|}{|\mathcal S|}.
\]
The complete subconfiguration $(P\cap B(L))\times(\mathcal S\cap\mathcal S(L))$ gives
\[
 I(P,\mathcal S)\ge(1-a)(1-b)|P||\mathcal S|
 \ge(1-a-b)|P||\mathcal S|.
\]
Hence $\mu\le a+b$, and taking the infimum proves \eqref{eq:model-distance-lower}.  The upper bound follows from \cref{thm:dense-freiman}.
\end{proof}

\begin{proposition}[The square-root deletion loss is sharp]\label{prop:dense-sqrt-sharp}
There are dense-endpoint systems in split dimension four for which
\[
 \mu(P,\mathcal S)\asymp q^{-1}
 \qquad\text{but}\qquad
 \operatorname{dist}_{\mathrm{lin}}(P,\mathcal S)\gg q^{-1/2}.
\]
Consequently the exponent $1/2$ in \eqref{eq:model-distance-upper} cannot be improved in general.
\end{proposition}

\begin{proof}
Use coordinates in which
\[
 Q(x)=x_1x_2+x_3x_4
\]
on $\mathbb F_q^4$, and let $e_1,e_3$ be the first and third coordinate vectors.  In lifted parameter space put
\[
 L_0=\{(te_1,0):t\in\mathbb F_q\},
 \qquad
 L_1=\{(te_1+ue_3,0):t,u\in\mathbb F_q\}.
\]
Then $L_0\subset L_1$ and a direct use of \eqref{eq:lambda-lift} gives
\begin{align*}
 B(L_0)&=\{x_2=0,\ x_3x_4=0\},\\
 B(L_1)&=\{x_2=x_4=0\}.
\end{align*}
Let $P_0=B(L_1)$, so $|P_0|=q^2$, and choose
\[
 P_1\subseteq\{x_2=x_3=0,\ x_4\ne0\}
\]
of cardinality $a=\lfloor q^{3/2}\rfloor$.  Let $\mathcal S_0=\mathcal S(L_0)$, so $|\mathcal S_0|=q$, and choose
\[
 \mathcal S_1\subseteq\mathcal S(L_1)\setminus\mathcal S(L_0)
\]
of cardinality $b=\lfloor q^{1/2}\rfloor$.  Put
\[
 P=P_0\sqcup P_1,
 \qquad
 \mathcal S=\mathcal S_0\sqcup\mathcal S_1.
\]
Every point of $P_0$ lies on every sphere of $\mathcal S$.  Every point of $P_1$ lies on every sphere of $\mathcal S_0$ and on no sphere of $\mathcal S_1$: indeed, if the lift is $(te_1+ue_3,0)$ with $u\ne0$, then for $x\in P_1$,
\[
 \Lambda_x(te_1+ue_3,0)=-\frac{u x_4}{2}\ne0.
\]
Thus the missing incidence graph is exactly the rectangle $P_1\times\mathcal S_1$, and
\[
 \mu(P,\mathcal S)
 =\frac{ab}{(q^2+a)(q+b)}\asymp q^{-1}.
\]

For any affine $L$ with $B(L)\ne\varnothing$, the complete product $B(L)\times\mathcal S(L)$ cannot contain both a point of $P_1$ and a sphere of $\mathcal S_1$.  Hence either $P_1\cap B(L)=\varnothing$ or $\mathcal S_1\cap\mathcal S(L)=\varnothing$.  Therefore
\[
 \operatorname{dist}_{\mathrm{lin}}(P,\mathcal S)
 \ge\min\left\{\frac a{q^2+a},\frac b{q+b}\right\}
 \asymp q^{-1/2},
\]
which proves the claim.
\end{proof}

\begin{corollary}[Coaxal stability under codimension-two nonconcentration]\label{cor:dense-coaxal}
Assume the hypotheses of \cref{thm:dense-freiman}, use the symmetric conclusion \eqref{eq:dense-symmetric}, and suppose that for some $\gamma>\delta$,
\begin{equation}\label{eq:codim-two-nonconcentration}
 |P\cap V|\le(1-\gamma)|P|
\end{equation}
for every affine flat $V\subseteq\mathbb F_q^d$ of codimension two.  Then the classifying space $L$ has dimension at most one.  If $(1-\delta)|\mathcal S|\ge2$, then $\dim L=1$: after deleting at most a $\delta$-fraction from each side, the remaining spheres lie in one coaxal pencil and the remaining points lie in its common sphere-pair section.
\end{corollary}

\begin{proof}
If $r=\dim L\ge2$, \eqref{eq:base-normal-form} places $B(L)$ inside an affine flat of codimension $r$, hence inside a codimension-two affine flat.  But $P_0\subseteq B(L)$ and $|P_0|\ge(1-\delta)|P|>(1-\gamma)|P|$, contradicting \eqref{eq:codim-two-nonconcentration}.  Thus $r\le1$.  If $\mathcal S_0$ contains at least two distinct spheres, then its affine span has dimension one, and \cref{lem:coaxal} identifies the model as a coaxal pencil.
\end{proof}

\begin{remark}[Why this is a classification theorem]\label{rem:dense-level3}
The conclusion is global rather than extractive.  It gives a complete list of exact ambient models, parametrized by affine subspaces $L\subseteq\mathbb F_q^{d+1}$ with nonempty base, proves that every dense-endpoint near-extremizer becomes a member of this list after deleting an explicit fraction from each side, and proves a converse estimate in the same edit metric.  For fixed $d$ there are only $q^{O_d(1)}$ possible ambient spaces $L$.  Arbitrary subsets of $B(L)$ and $\mathcal S(L)$ are allowed, just as a Freiman theorem places a set inside a structured ambient progression rather than prescribing every element of the set.
\end{remark}

\subsection{Affine singular pencils}

An affine flat $A=a+W$ has totally singular direction if $Q(w)=0$ for every $w\in W$.  Polarization gives $W\subseteq W^{\perp_Q}$.  Define the affine singular pencil through $A$ by
\begin{equation}\label{eq:affine-pencil}
 \mathcal P(A)=
 \{S_Q(c,Q(a-c)):c-a\in W^{\perp_Q}\}.
\end{equation}
If $x=a+w\in A$ and $c-a\in W^{\perp_Q}$, then
\[
 Q(x-c)=Q(w-(c-a))=Q(c-a)=Q(a-c),
\]
so every member of $\mathcal P(A)$ contains all of $A$.

\begin{example}[One pencil and the adaptive scale]\label{ex:one-pencil}
Let $P\subseteq A$ have $n$ points and let $\mathcal S\subseteq\mathcal P(A)$ have $s$ members.  Then
\[
 I(P,\mathcal S)=ns,
 \qquad
 \sigma(P,\mathcal S)=(1-q^{-1})ns.
\]
If $K$ is the exact parameter in \eqref{eq:intro-fixed}, then
\begin{equation}\label{eq:pencil-sharp}
 \frac{K^2q^{d-1}}{s}=(1-q^{-1})^2n.
\end{equation}
Thus the scale $K^2q^{d-1}/s$ in \cref{cor:fixed-adaptive} is optimal up to an absolute constant.

When $d\ge3$, nondegenerate quadratic spaces over $\mathbb F_q$ are isotropic, so one may take $W$ to be a singular line.  The nonzero-radius part of such a pencil has $\gg q^{d-1}$ members, and every nonempty fiber at an attained nonzero radius has $\asymp_d q^{d-2}$ members.  Thus the mixed-radius scale $q^{d/2}$ is necessary, while fixed-radius slices attain the $q^{(d-1)/2}$ scale up to constants.

If $d=2m+1$ and $W$ is maximal totally singular, then $W^{\perp_Q}/W$ is a nondegenerate one-dimensional quadratic space.  Hence $(q-1)q^m$ members of the pencil have nonzero radius.  Choosing admissible integers $n\le q^m$ and $s\le(q-1)q^m$ with $s\asymp C_1^2n$ gives
\[
 \frac{Kq^{(d-1)/2}}{C_1}\asymp n,
\]
so the $K/C_1$ dependence in the moderate theorem is optimal throughout the range realized by the pencil.

If $d=2m$ and $Q$ is split, a maximal totally singular $W$ satisfies $W^{\perp_Q}=W$.  All members of the corresponding maximal pencil then have radius zero.  Taking $P=W$ and $s\asymp q^m$ gives $K\asymp q^{1/2}$ and $Kq^{(d-1)/2}\asymp|P|$, including the split planar case.  The fixed-nonzero-radius block examples of \cref{sec:atomicity} show that radius zero is not responsible for the failure of proportional concentration in dimensions at least four.
\end{example}

\section{Scale-adaptive inverse theorems}\label{sec:inverse}

\subsection{Exact energy extraction}

For $P$ and $\mathcal S$, define the point-degree energy
\[
 E(P,\mathcal S)=\sum_{p\in P}d_{\mathcal S}(p)^2.
\]
Expanding the square gives
\begin{equation}\label{eq:energy-decomp}
 E(P,\mathcal S)=I(P,\mathcal S)+E_{\mathrm{off}}(P,\mathcal S),
\end{equation}
where
\[
 E_{\mathrm{off}}(P,\mathcal S)=
 \sum_{\substack{S,S'\in\mathcal S\\S\ne S'}}
 |P\cap S\cap S'|.
\]

\begin{theorem}[Exact scale-adaptive extraction]\label{thm:exact-adaptive}
Let $P\subseteq\mathbb F_q^d$ be nonempty and let $\mathcal S$ be a nonempty family of distinct spheres.  Put
\[
 n=|P|,\qquad s=|\mathcal S|,\qquad I=I(P,\mathcal S),
 \qquad \mathcal E_0=\frac{I^2}{n}-I.
\]
Assume $\mathcal E_0>0$, and let $0<\eta<1$.  Define
\[
 \lambda_\eta=\eta\frac{\mathcal E_0}{s^2}
\]
and let $\mathcal R_\eta$ be the set of ordered pairs $(S,S')\in\mathcal S^2$, $S\ne S'$, with
\[
 |P\cap S\cap S'|\ge\lambda_\eta.
\]
Then
\begin{equation}\label{eq:exact-rich-energy}
 \sum_{(S,S')\in\mathcal R_\eta}|P\cap S\cap S'|
 \ge(1-\eta)\mathcal E_0.
\end{equation}
Every pair in $\mathcal R_\eta$ is nonconcentric, generates a canonical coaxal pencil whose base locus contains at least $\lambda_\eta$ points of $P$, and has a radical hyperplane containing at least the same number.  Moreover,
\begin{equation}\label{eq:exact-rich-count}
 |\mathcal R_\eta|\ge
 \frac{(1-\eta)\mathcal E_0}{q^{d-1}},
\end{equation}
and, if $d\ge3$,
\begin{equation}\label{eq:exact-rich-count-d3}
 |\mathcal R_\eta|\ge
 \frac{(1-\eta)\mathcal E_0}{2q^{d-2}}.
\end{equation}
\end{theorem}

\begin{proof}
Cauchy--Schwarz gives $E(P,\mathcal S)\ge I^2/n$, hence by \eqref{eq:energy-decomp},
\begin{equation}\label{eq:Eoff-E0}
 E_{\mathrm{off}}(P,\mathcal S)\ge\mathcal E_0.
\end{equation}
Pairs outside $\mathcal R_\eta$ contribute less than
\[
 \lambda_\eta s^2=\eta\mathcal E_0.
\]
Subtracting from \eqref{eq:Eoff-E0} proves \eqref{eq:exact-rich-energy}.  A contributing pair has nonempty intersection and therefore is nonconcentric.  The radical-hyperplane conclusion follows from \cref{lem:radical}.  Dividing \eqref{eq:exact-rich-energy} by $q^{d-1}$ gives \eqref{eq:exact-rich-count}; for $d\ge3$, division by $2q^{d-2}$ gives \eqref{eq:exact-rich-count-d3}.
\end{proof}

\subsection{Fixed-scale and moderate consequences}

\begin{corollary}[Fixed-scale adaptive inverse theorem]\label{cor:fixed-adaptive}
Let $P\subseteq\mathbb F_q^d$ be nonempty, let $\mathcal S$ be a nonempty family of distinct spheres, and let $K>0$.  Suppose
\begin{equation}\label{eq:fixed-surplus}
 \sigma(P,\mathcal S)\ge Kq^{(d-1)/2}\sqrt{|P||\mathcal S|}
\end{equation}
and
\begin{equation}\label{eq:point-sensitive-lower}
 K^2q^{d-1}|\mathcal S|\ge4|P|.
\end{equation}
Then at least $K^2|\mathcal S|/4$ ordered pairs satisfy
\begin{equation}\label{eq:adaptive-section}
 |P\cap S\cap S'|\ge
 \frac{K^2q^{d-1}}{4|\mathcal S|}.
\end{equation}
Every such pair supplies a radical hyperplane containing at least the same number of points.  At least $K^2/8$ spheres each have at least $K^2/8$ witnessing partners.  If $d\ge3$, at least $K^2q|\mathcal S|/8$ ordered pairs satisfy \eqref{eq:adaptive-section}, and at least $K^2q/16$ spheres each have at least $K^2q/16$ witnessing partners.
\end{corollary}

\begin{proof}
Put $n=|P|$, $s=|\mathcal S|$, and $\Sigma=\sigma(P,\mathcal S)$.  By \eqref{eq:fixed-surplus} and \eqref{eq:point-sensitive-lower}, $\Sigma\ge2n$.  Since $I\ge\Sigma$ and $x\mapsto x^2/n-x$ is increasing for $x\ge n/2$,
\[
 \mathcal E_0\ge\frac{\Sigma^2}{n}-\Sigma
 \ge\frac{\Sigma^2}{2n}
 \ge\frac12K^2q^{d-1}s.
\]
Apply \cref{thm:exact-adaptive} with $\eta=1/2$.  Its threshold is at least $K^2q^{d-1}/(4s)$, and its rich energy is at least $K^2q^{d-1}s/4$.  Division by the pair bounds in \cref{lem:radical} gives the asserted ordered-pair counts.

For the partner statements, let $f(S)$ be the number of witnessing partners of $S$.  If $\sum_S f(S)\ge as$, then vertices with $f(S)<a/2$ contribute less than $as/2$, while each remaining vertex contributes at most $s$; hence at least $a/2$ vertices have at least $a/2$ partners.  Use $a=K^2/4$, and, when $d\ge3$, $a=K^2q/8$.
\end{proof}

Since $|P|\le q^d$, the simpler condition $|\mathcal S|\ge4q/K^2$ implies \eqref{eq:point-sensitive-lower}.

\begin{corollary}[Moderate-family inverse theorem]\label{cor:moderate}
Let $P\subseteq\mathbb F_q^d$ be nonempty, let $\mathcal S$ be a nonempty family of distinct spheres, let $K>0$, and let $C_1\ge1$.  Suppose \eqref{eq:fixed-surplus} holds and
\begin{equation}\label{eq:moderate-window}
 \frac{4|P|}{K^2q^{d-1}}
 \le|\mathcal S|
 \le C_1Kq^{(d-1)/2}.
\end{equation}
Then at least $K^2|\mathcal S|/4$ ordered pairs of spheres satisfy
\begin{equation}\label{eq:moderate-section}
 |P\cap S\cap S'|\ge
 \frac{Kq^{(d-1)/2}}{4C_1}.
\end{equation}
Their radical hyperplanes contain at least the same number of points.  If $d\ge3$, the number of witnessing pairs is at least $K^2q|\mathcal S|/8$.  The partner-abundance conclusions of \cref{cor:fixed-adaptive} hold as well.
\end{corollary}

\begin{proof}
Apply \cref{cor:fixed-adaptive}; the upper bound in \eqref{eq:moderate-window} turns \eqref{eq:adaptive-section} into \eqref{eq:moderate-section}.
\end{proof}

\begin{corollary}[Stability]\label{cor:stability}
Let $P\subseteq\mathbb F_q^d$ be nonempty, let $\mathcal S$ be a nonempty family of distinct spheres, let $K>0$, and let $C_1\ge1$.  Assume \eqref{eq:moderate-window}.  If every pair of distinct spheres in $\mathcal S$ satisfies
\[
 |P\cap S\cap S'|<\frac{Kq^{(d-1)/2}}{4C_1},
\]
then \eqref{eq:fixed-surplus} fails.  It is enough to impose the same strict upper bound on $|P\cap H|$ for every affine hyperplane $H$.
\end{corollary}

\begin{proof}
This is the contrapositive of \cref{cor:moderate}, together with the containment of every nonconcentric sphere-pair section in its radical hyperplane.
\end{proof}

\subsection{Exact sphere-complement duality and its limited omitted-family consequence}

The complete mixed-radius design has an exact complement symmetry.  The following auxiliary proposition records its consequence for deficits.  As quantified below, the pair-section threshold is structurally nontrivial only when the omitted family is very small.

\begin{proposition}[Complement duality and an omitted-family corollary]\label{prop:deficit-complement}
Let $P\subseteq\mathbb F_q^d$ be nonempty, put $n=|P|$, and let $\mathcal S\subsetneq\Omega$.  Write
\[
 D=\frac{n|\mathcal S|}{q}-I(P,\mathcal S)=-\sigma(P,\mathcal S)>0,
 \qquad
 \overline{\mathcal S}=\Omega\setminus\mathcal S,
 \qquad
 \bar s=|\overline{\mathcal S}|.
\]
Then
\begin{equation}\label{eq:deficit-duality}
 \sigma(P,\overline{\mathcal S})=D.
\end{equation}
If $D\ge2n$, then at least
\begin{equation}\label{eq:deficit-count}
 \frac{D^2}{4nq^{d-1}}
\end{equation}
ordered pairs of distinct omitted spheres $T,T'\in\overline{\mathcal S}$ satisfy
\begin{equation}\label{eq:deficit-section}
 |P\cap T\cap T'|
 \ge\frac{D^2}{4n\bar s^2}.
\end{equation}
If $d\ge3$, the number of such ordered pairs is at least
\begin{equation}\label{eq:deficit-count-d3}
 \frac{D^2}{8nq^{d-2}}.
\end{equation}
Equivalently, if for some $K>0$
\begin{equation}\label{eq:deficit-fixed-scale}
 \sigma(P,\mathcal S)
 \le-Kq^{(d-1)/2}\sqrt{n|\mathcal S|}
 \quad\text{and}\quad
 K^2q^{d-1}|\mathcal S|\ge4n,
\end{equation}
then the omitted family contains at least $K^2|\mathcal S|/4$ ordered pairs satisfying
\begin{equation}\label{eq:deficit-fixed-section}
 |P\cap T\cap T'|
 \ge\frac{K^2q^{d-1}|\mathcal S|}{4\bar s^2},
\end{equation}
and at least $K^2q|\mathcal S|/8$ such pairs when $d\ge3$.
\end{proposition}

\begin{proof}
Every point lies on $q^d$ members of $\Omega$, so $\sigma(P,\Omega)=0$ and therefore
\[
 \sigma(P,\overline{\mathcal S})=D.
\]
Put $\bar I=I(P,\overline{\mathcal S})$.  Since $\bar I\ge D\ge2n$ and $x\mapsto x^2/n-x$ is increasing for $x\ge n/2$,
\[
 \frac{\bar I^2}{n}-\bar I
 \ge\frac{D^2}{n}-D
 \ge\frac{D^2}{2n}.
\]
Apply \cref{thm:exact-adaptive} to $(P,\overline{\mathcal S})$ with $\eta=1/2$.  The threshold is at least $D^2/(4n\bar s^2)$, and the rich energy is at least $D^2/(4n)$.  Division by the pair bounds $q^{d-1}$ and, for $d\ge3$, $2q^{d-2}$ proves \eqref{eq:deficit-count}--\eqref{eq:deficit-count-d3}.  Under \eqref{eq:deficit-fixed-scale}, one has $D^2\ge K^2q^{d-1}n|\mathcal S|\ge4n^2$, so the first part applies and gives the normalized conclusions.
\end{proof}

\begin{remark}[Nontriviality range]\label{rem:deficit-scope}
The proposition is an exact statement about the complete mixed-radius design, but it is usually formal rather than structural.  By \eqref{eq:complement-sphere}, with $s=|\mathcal S|$ and $M=|\Omega|=q^{d+1}$,
\[
 D^2\le(q^d-q^{d-1})ns\left(1-\frac{s}{M}\right)
 =\frac{q-1}{q^2}\,ns\bar s.
\]
Consequently
\begin{equation}\label{eq:deficit-nontriviality}
 \frac{D^2}{4n\bar s^2}
 \le \frac{q-1}{4q^2}\frac{s}{\bar s}.
\end{equation}
In particular, if $s\le M/2$, then the section-size threshold in \eqref{eq:deficit-section} is strictly smaller than $1$, so the conclusion gives no more than nonemptiness.  For that threshold to exceed $1$, it is necessary that
\[
 \frac{s}{\bar s}>\frac{4q^2}{q-1};
\]
thus the omitted family must be an $O(q^{-1})$ fraction of the full sphere system.  The pair-count bound is also nonstructural in a large natural range.  If $s\le q^d/2$, then every point belongs to at least $q^d-s$ omitted spheres, so one point alone supplies at least
\[
 (q^d-s)(q^d-s-1)
\]
ordered omitted pairs with nonempty $P$-section.  Meanwhile \eqref{eq:complement-sphere} gives
\[
 \frac{D^2}{4nq^{d-1}}
 \le \frac{q-1}{4}s
 \le \frac{q-1}{8}q^d
 <\frac{q^d}{2}\left(\frac{q^d}{2}-1\right)
\]
for $q\ge3$ and $d\ge2$.  Hence the stated pair count is weaker there than this trivial one-point count.  We therefore use \cref{prop:deficit-complement} only as a complement-duality observation, not as an intrinsic inverse theorem for deficits in the original family.  A deficit inside one prescribed fixed-radius class depends on near-extremal Kloosterman and Sali\'e Fourier directions and is not resolved here.
\end{remark}

\subsection{A dense rich-pair core in dimension three}

\begin{proposition}[Positive-density rich-pair core in $d=3$]\label{prop:d3-core}
Assume the hypotheses of \cref{cor:moderate} with $d=3$, and put $s=|\mathcal S|$.  Let $G$ be the simple graph on $\mathcal S$ in which $S$ and $S'$ are adjacent when
\[
 |P\cap S\cap S'|\ge\frac{Kq}{4C_1}.
\]
Then
\begin{equation}\label{eq:d3-density}
 e(G)\ge\frac{K^2qs}{16}\ge\frac{K}{16C_1}s^2.
\end{equation}
Moreover, there is a nonempty induced subgraph $G[\mathcal S_0]$ satisfying
\begin{equation}\label{eq:d3-core}
 \delta(G[\mathcal S_0])\ge\frac{K}{16C_1}s,
 \qquad
 |\mathcal S_0|\ge1+\frac{K}{16C_1}s.
\end{equation}
Thus a positive proportion of the sphere family lies in a subfamily in which every sphere has a positive proportion of rich partners.
\end{proposition}

\begin{proof}
By \cref{cor:moderate}, at least $K^2qs/8$ ordered pairs meet the displayed threshold.  Each edge of $G$ contributes the two orientations, so $e(G)\ge K^2qs/16$.  Since $s\le C_1Kq$, the second inequality in \eqref{eq:d3-density} follows.

Set $a=e(G)/s$.  Repeatedly delete a vertex of current degree less than $a$.  If every vertex were deleted, the sum of the degrees at deletion would be strictly less than $sa=e(G)$, although each original edge is counted exactly once when its first endpoint is deleted.  Hence a nonempty induced subgraph remains and has minimum degree at least $a\ge Ks/(16C_1)$.  Its number of vertices is at least one more than its minimum degree.
\end{proof}

\subsection{A log-free incidence-rich refinement}

\begin{proposition}[High-degree rich section without logarithmic loss]\label{prop:logfree}
Let $P\subseteq\mathbb F_q^d$ be nonempty, let $\mathcal S$ be a nonempty family of distinct spheres, let $K>0$, and let $C_1\ge1$.  Suppose \eqref{eq:fixed-surplus} holds and
\begin{equation}\label{eq:logfree-window}
 \frac{64|P|}{K^2q^{d-1}}
 \le|\mathcal S|
 \le C_1Kq^{(d-1)/2}.
\end{equation}
Put $n=|P|$, $s=|\mathcal S|$, $I=I(P,\mathcal S)$, and
\[
 M=\left\lceil\frac{I}{2n}\right\rceil.
\]
There are two nonconcentric spheres $S_0,S_1\in\mathcal S$, a set
\[
 P'\subseteq P\cap S_0\cap S_1,
\]
and a subfamily $\mathcal S'\subseteq\mathcal S$ such that
\begin{enumerate}[label=\textup{(\arabic*)}]
\item $|P'|\ge Kq^{(d-1)/2}/(8C_1)$;
\item $d_{\mathcal S}(p)\ge M$ for every $p\in P'$;
\item every $S\in\mathcal S'$ satisfies
\[
 |P'\cap S|\ge\frac{M}{2s}|P'|,
\]
and
\[
 \sum_{S\in\mathcal S'}|P'\cap S|\ge\frac M2|P'|,
 \qquad |\mathcal S'|\ge\frac M2.
\]
\end{enumerate}
\end{proposition}

\begin{proof}
Let
\[
 A=\{p\in P:d_{\mathcal S}(p)\ge M\}.
\]
Since $M=\lceil I/(2n)\rceil$, points outside $A$ have degree strictly less than $I/(2n)$ and contribute less than $I/2$ incidences.  Hence
\begin{equation}\label{eq:IA-half}
 I(A,\mathcal S)>\frac I2.
\end{equation}
The lower bound in \eqref{eq:logfree-window} gives
\[
 M\ge\frac{I}{2n}
 \ge\frac{Kq^{(d-1)/2}}2\sqrt{\frac{s}{n}}
 \ge4.
\]
On $A$, $d_{\mathcal S}(p)^2\ge M d_{\mathcal S}(p)$, so
\[
 E(A,\mathcal S)
 \ge M I(A,\mathcal S)
 >\frac{I^2}{4n}
 \ge\frac{K^2q^{d-1}s}{4}.
\]
Also $I(A,\mathcal S)\le E(A,\mathcal S)/M\le E(A,\mathcal S)/4$, and therefore
\begin{equation}\label{eq:off-A}
 E_{\mathrm{off}}(A,\mathcal S)
 \ge\frac34E(A,\mathcal S)
 >\frac{3K^2q^{d-1}s}{16}.
\end{equation}
Set $\lambda=Kq^{(d-1)/2}/(8C_1)$.  Pairs with $|A\cap S\cap S'|<\lambda$ carry at most
\[
 \lambda s^2\le\frac{K^2q^{d-1}s}{8},
\]
which is smaller than \eqref{eq:off-A}.  Thus a nonconcentric pair $S_0,S_1$ satisfies
\[
 |A\cap S_0\cap S_1|\ge\lambda.
\]
Take $P'=A\cap S_0\cap S_1$.

Finally, let
\[
 \mathcal S'=\left\{S\in\mathcal S:
 |P'\cap S|\ge\frac{M}{2s}|P'|\right\}.
\]
Since $I(P',\mathcal S)\ge M|P'|$, spheres outside $\mathcal S'$ contribute less than $M|P'|/2$.  Thus the asserted incidence lower bound holds, and division by the trivial upper bound $|P'|$ for each summand gives $|\mathcal S'|\ge M/2$.
\end{proof}

\subsection{Decomposition of the positive surplus}

Fix $\mathcal S$ and write
\[
 e_{\mathcal S}(p)=d_{\mathcal S}(p)-\frac{|\mathcal S|}{q},
 \qquad
 P_{\mathrm{pos}}=\{p\in P:e_{\mathcal S}(p)>0\}.
\]
For $A\subseteq P$, abbreviate $\sigma(A)=\sum_{p\in A}e_{\mathcal S}(p)$.

\begin{theorem}[Positive-surplus decomposition into section-supported pieces]\label{thm:decomposition}
Let $P\subseteq\mathbb F_q^d$ be nonempty, let $\mathcal S$ be a nonempty family of distinct spheres, and let $K>0$.  Assume \eqref{eq:fixed-surplus}.  Let $0<\eps\le1$ and suppose
\begin{equation}\label{eq:decomp-lower}
 \eps^2K^2q^{d-1}|\mathcal S|\ge4|P|.
\end{equation}
Then $P$ has a partition
\[
 P=R_1\sqcup\cdots\sqcup R_T\sqcup P_{\mathrm{rem}}
\]
with the following properties.
\begin{enumerate}[label=\textup{(\arabic*)}]
\item For every $j$, there is a nonconcentric pair $(S_j,S_j')\in\mathcal S^2$ such that
\begin{equation}\label{eq:piece-containment}
 R_j\subseteq P_{\mathrm{pos}}\cap S_j\cap S_j',
 \qquad
 |R_j|\ge
 \frac{\eps^2K^2q^{d-1}}{4|\mathcal S|}.
\end{equation}
\item Every extracted piece has strictly positive surplus: $\sigma(R_j)>0$.
\item The number of pieces satisfies
\[
 T\le\frac{4|P||\mathcal S|}{\eps^2K^2q^{d-1}}.
\]
\item The remainder obeys
\[
 \sigma(P_{\mathrm{rem}})<\eps\sigma(P,\mathcal S),
\]
and the sections carry the remaining net surplus:
\[
 \sum_{j=1}^T\sigma(R_j)>(1-\eps)\sigma(P,\mathcal S).
\]
\end{enumerate}
If $|\mathcal S|\le C_1Kq^{(d-1)/2}$, then every piece has size at least
\[
 \frac{\eps^2Kq^{(d-1)/2}}{4C_1}.
\]
\end{theorem}

\begin{proof}
Put $s=|\mathcal S|$, $n=|P|$, and $\Sigma_0=\sigma(P,\mathcal S)>0$.  Split
\[
 P=P_{\mathrm{pos}}\sqcup P_{\mathrm{non}},
 \qquad
 P_{\mathrm{non}}=\{p\in P:e_{\mathcal S}(p)\le0\}.
\]
Then $\sigma(P_{\mathrm{non}})\le0$, so $\sigma(P_{\mathrm{pos}})\ge\Sigma_0$.

Start with $A_0=P_{\mathrm{pos}}$.  Given $A_t\subseteq P_{\mathrm{pos}}$, continue while
\[
 \sigma(A_t)\ge\eps\Sigma_0.
\]
By \eqref{eq:fixed-surplus} and \eqref{eq:decomp-lower},
\[
 \sigma(A_t)\ge\eps\Sigma_0
 \ge\eps Kq^{(d-1)/2}\sqrt{ns}
 \ge2n\ge2|A_t|.
\]
Define the exact residual parameter
\[
 K_t=\frac{\sigma(A_t)}{q^{(d-1)/2}\sqrt{|A_t|s}}.
\]
The inequality $\sigma(A_t)\ge2|A_t|$ is equivalent to $K_t^2q^{d-1}s\ge4|A_t|$, so \cref{cor:fixed-adaptive} applies to $(A_t,\mathcal S)$ and gives a pair whose common section in $A_t$ has size at least
\[
 \frac{K_t^2q^{d-1}}{4s}
 =\frac{\sigma(A_t)^2}{4|A_t|s^2}
 \ge\frac{\eps^2K^2q^{d-1}}{4s}.
\]
Set
\[
 R_{t+1}=A_t\cap S_{t+1}\cap S_{t+1}',
 \qquad A_{t+1}=A_t\setminus R_{t+1}.
\]
Every point of $A_t$ has positive excess, hence $\sigma(R_{t+1})>0$.  The pieces are disjoint and have the lower bound in \eqref{eq:piece-containment}, so the asserted bound on $T$ follows.

Let $A_T$ be the final positive-excess remainder and put
\[
 P_{\mathrm{rem}}=A_T\sqcup P_{\mathrm{non}}.
\]
At termination $\sigma(A_T)<\eps\Sigma_0$, and therefore
\[
 \sigma(P_{\mathrm{rem}})
 =\sigma(A_T)+\sigma(P_{\mathrm{non}})
 <\eps\Sigma_0.
\]
Additivity over the partition gives the final surplus inequality.  The moderate-family estimate follows immediately from the upper cap.
\end{proof}

\begin{remark}\label{rem:decomp-converse}
The geometry in \cref{thm:decomposition} lies entirely in the necessity direction.  Conversely, for any partition of $P$, the identity
\[
 \sigma(P,\mathcal S)=\sum_j\sigma(R_j)+\sigma(P_{\mathrm{rem}})
\]
is only additivity.  We therefore do not describe that bookkeeping implication as a structural converse.
\end{remark}

\begin{lemma}[Surplus assembly criterion]\label{lem:assembly}
Let $P=R_1\sqcup\cdots\sqcup R_T\sqcup P_{\mathrm{rem}}$ be any partition, and fix a nonempty family $\mathcal S$ of distinct spheres.  Suppose that, for some $K>0$ and $0\le\theta<1$,
\[
 \sum_{j=1}^T\sigma(R_j,\mathcal S)
 \ge Kq^{(d-1)/2}\sqrt{|P||\mathcal S|}
\]
and
\[
 \sigma(P_{\mathrm{rem}},\mathcal S)
 \ge-\theta Kq^{(d-1)/2}\sqrt{|P||\mathcal S|}.
\]
Then $(P,\mathcal S)$ has a $(1-\theta)K$-large fixed-scale surplus.  A sufficient pointwise condition for the first inequality is the existence of numbers $\delta_j>0$ such that
\[
 d_{\mathcal S}(p)\ge\frac{|\mathcal S|}{q}+\delta_j
 \quad(p\in R_j),
 \qquad
 \sum_{j=1}^T\delta_j|R_j|
 \ge Kq^{(d-1)/2}\sqrt{|P||\mathcal S|}.
\]
\end{lemma}

\begin{proof}
Additivity of $\sigma(\,\cdot\,,\mathcal S)$ over disjoint point sets gives the first assertion.  Under the pointwise condition,
\[
 \sigma(R_j,\mathcal S)
 =\sum_{p\in R_j}\left(d_{\mathcal S}(p)-\frac{|\mathcal S|}{q}\right)
 \ge\delta_j|R_j|;
\]
summing proves the stated sufficient condition.
\end{proof}

\section{Atomicity obstructions}\label{sec:atomicity}

The decomposition theorem cannot in general be replaced by one proportional section, nor can a dense rich-pair graph be forced into one coaxal pencil or a centered direct sum.  We give two structurally different obstruction classes.  The first consists of exact centered blocks at one fixed nonzero radius.  The second consists of diffuse root-evaluation designs carried by rational normal curves in lifted parameter space.

\subsection{The isotropic quotient}

Let $d\ge3$.  Since every nondegenerate quadratic form in at least three variables over a finite field is isotropic, choose $v\ne0$ with $Q(v)=0$.  Choose $w\in\mathbb F_q^d$ with
\[
 \ip{w}{v}=1.
\]
The restriction of $Q$ to $v^{\perp_Q}$ has radical $\mathbb F_q v$ and therefore induces a nondegenerate quadratic form $\overline Q$ on
\[
 \overline V=v^{\perp_Q}/\mathbb F_q v,
 \qquad \dim\overline V=d-2.
\]
Choose a represented value $R\in\mathbb F_q^\times$ and put
\begin{equation}\label{eq:NR-def}
 X_R=\{y\in v^{\perp_Q}:Q(y)=R\},
 \qquad N_R=|X_R|.
\end{equation}
Because $Q(y+uv)=Q(y)$ on $v^{\perp_Q}$, every point of the quotient level set has exactly $q$ lifts, so
\begin{equation}\label{eq:NR-quotient}
 N_R=q\,|\{\bar y\in\overline V:\overline Q(\bar y)=R\}|.
\end{equation}
The standard evaluation of quadratic level sets gives
\begin{equation}\label{eq:NR-asymptotic}
 N_R=q^{d-2}+O_d\bigl(q^{(d-1)/2}\bigr)
 \qquad(d\ge4).
\end{equation}
When $d=3$, $R$ may be chosen so that $N_R=2q$.  We use only that $N_R\asymp_d q^{d-2}$ for $d\ge4$.  See, for example, \cite[Chapter~6]{LN}.

\subsection{Exact centered-incidence blocks}

Let $Z\subseteq\mathbb F_q$ be nonempty, with $|Z|=t$.  For $z\in Z$, define
\begin{equation}\label{eq:fixed-block-def}
 c_z=zw,
 \qquad
 A_z=c_z+X_R,
 \qquad
 \mathcal S_z=\{S_Q(c_z+uv,R):u\in\mathbb F_q\}.
\end{equation}
Finally put
\begin{equation}\label{eq:fixed-global-def}
 P=\bigsqcup_{z\in Z}A_z,
 \qquad
 \mathcal S=\bigsqcup_{z\in Z}\mathcal S_z.
\end{equation}
The unions are disjoint: every $x\in A_z$ satisfies $\ip{x}{v}=z$, and the centers $zw+uv$ are distinct as $(z,u)$ varies.  Every sphere in $\mathcal S$ has the same nonzero radius $R$.

\begin{theorem}[Exact fixed-radius block system]\label{thm:fixed-block}
For the configuration in \eqref{eq:fixed-block-def}--\eqref{eq:fixed-global-def}, the following hold.
\begin{enumerate}[label=\textup{(\arabic*)}]
\item The sizes are
\begin{equation}\label{eq:fixed-sizes}
 |P|=tN_R,
 \qquad
 |\mathcal S|=tq.
\end{equation}
\item The centered block-incidence matrix is diagonal:
\begin{equation}\label{eq:fixed-cross}
 \sigma(A_b,\mathcal S_z)=
 \begin{cases}
  (q-1)N_R,&b=z,\\
  0,&b\ne z.
 \end{cases}
\end{equation}
Consequently
\begin{equation}\label{eq:fixed-sigma}
 \sigma(P,\mathcal S)=t(q-1)N_R
 =\sum_{z\in Z}\sigma(A_z,\mathcal S_z).
\end{equation}
The exact parameter in the fixed-scale normalization \eqref{eq:intro-fixed} is
\begin{equation}\label{eq:fixed-K}
 K_*=(1-q^{-1})\frac{\sqrt{N_R}}{q^{(d-2)/2}}.
\end{equation}
Thus $K_*\to1$ for $d\ge4$, while in $d=3$ one may take
$K_*=(1-q^{-1})\sqrt2$.
\item Every point of $P$ has degree
\begin{equation}\label{eq:fixed-degree}
 d_{\mathcal S}(x)=q+t-1.
\end{equation}
\item If $u\ne u'$, then
\begin{equation}\label{eq:fixed-common-section}
 S_Q(c_z+uv,R)\cap S_Q(c_z+u'v,R)=A_z.
\end{equation}
All these pairs have center-difference direction $[v]$, and their radical hyperplane is
\begin{equation}\label{eq:fixed-Hz}
 H_z=\{x\in\mathbb F_q^d:\ip{x}{v}=z\}.
\end{equation}
\item If $d\ge4$ and $H$ is an affine hyperplane, then
\begin{equation}\label{eq:fixed-hyperplane-exact}
 |P\cap H|\le
 \max\{N_R,\,2tq^{d-3}\}.
\end{equation}
In particular, there is a constant $B_d>0$ such that, for all sufficiently large $q$,
\begin{equation}\label{eq:fixed-hyperplane-ratio}
 \max_H\frac{|P\cap H|}{|P|}
 \le B_d\max\left\{\frac1t,\frac1q\right\}.
\end{equation}
\end{enumerate}
\end{theorem}

\begin{proof}
The size identities follow from $|A_z|=N_R$ and $|\mathcal S_z|=q$.

Fix $x=c_b+y\in A_b$, where $y\in v^{\perp_Q}$ and $Q(y)=R$, and consider a sphere $S_Q(c_z+uv,R)$.  Put $a=b-z$.  Since $Q(v)=0$, $\ip{y}{v}=0$, and $\ip{w}{v}=1$,
\begin{align}
 Q(x-c_z-uv)
 &=Q(aw+y-uv)\notag\\
 &=R+a^2Q(w)+2a\ip{w}{y}-2au.
 \label{eq:fixed-membership}
\end{align}
If $b=z$, the right-hand side is $R$ for every $u$, so every sphere of $\mathcal S_z$ contains all of $A_z$.  If $b\ne z$, equation \eqref{eq:fixed-membership} equals $R$ for exactly one value of $u$.  Hence
\[
 I(A_b,\mathcal S_z)=
 \begin{cases}
 qN_R,&b=z,\\
 N_R,&b\ne z,
 \end{cases}
\]
and subtracting the baseline $|A_b||\mathcal S_z|/q=N_R$ proves \eqref{eq:fixed-cross}.  Summation proves \eqref{eq:fixed-sigma}.  Substitution of \eqref{eq:fixed-sizes} into \eqref{eq:intro-fixed} gives \eqref{eq:fixed-K}.  The asymptotics follow from \eqref{eq:NR-asymptotic}, and the degree formula follows from the same membership calculation: a point lies on all $q$ spheres in its own block and exactly one sphere in each of the other $t-1$ blocks.

For $u\ne u'$, subtracting the two sphere equations gives $\ip{x}{v}=z$.  Hence any point in their ambient intersection can be written $x=c_z+y$ with $y\in v^{\perp_Q}$; either sphere equation then reduces to $Q(y)=R$.  Thus the ambient intersection is exactly $A_z$, proving \eqref{eq:fixed-common-section}, and the same subtraction proves \eqref{eq:fixed-Hz}.

It remains to prove the hyperplane bound.  Write $H=\{x:\ell(x)=a_0\}$ with $\ell\ne0$.  If $\ell$ vanishes on $v^{\perp_Q}$, then $\ker\ell=v^{\perp_Q}$.  Since $\ell(w)\ne0$, the values of $\ell$ on the blocks $A_z$ are distinct, so $H$ meets at most one block and $|P\cap H|\le N_R$.

Suppose that $\ell|_{v^{\perp_Q}}\ne0$.  Fix $z$.  If $\ell(v)\ne0$, every quotient class in $\overline V$ has exactly one representative satisfying the affine equation defining $H-c_z$; hence
\[
 |A_z\cap H|= |\{\bar y:\overline Q(\bar y)=R\}|=N_R/q\le2q^{d-3},
\]
where the final inequality is the Schwartz--Zippel bound for the nonzero polynomial $\overline Q-R$.  If $\ell(v)=0$, then $\ell$ descends to a nonzero linear functional $\bar\ell$ on $\overline V$.  The restriction of $\overline Q-R$ to the affine hyperplane $\bar\ell=\text{constant}$ is a nonzero polynomial.  Indeed, if it vanished identically, then $\overline Q$ would vanish on $\ker\bar\ell$.  When $\dim\overline V\ge3$, this contradicts the Witt-index bound.  When $\dim\overline V=2$, an affine line contained in the nonzero level $\overline Q=R$ would have isotropic direction $L$ and base point in $L^{\perp_Q}=L$, forcing $R=0$.  Thus Schwartz--Zippel gives at most $2q^{d-4}$ quotient points, each with $q$ lifts along $\mathbb F_q v$, and again
\[
 |A_z\cap H|\le2q^{d-3}.
\]
Summing over $z$ proves \eqref{eq:fixed-hyperplane-exact}.  Finally \eqref{eq:NR-asymptotic} gives $N_R\asymp_d q^{d-2}$ for large $q$, and \eqref{eq:fixed-hyperplane-ratio} follows.
\end{proof}

\subsection{Consequences in the moderate range}

\begin{corollary}[Fixed-radius block obstructions in the family-only range]\label{cor:fixed-no-proportional}
The exact systems of \cref{thm:fixed-block} meet the point-independent lower condition $|\mathcal S|\ge4q/K_*^2$ as follows.
\begin{enumerate}[label=\textup{(\arabic*)}]
\item Fix $d\ge4$ and $C_1\ge1$.  For all sufficiently large odd $q$, every nondegenerate quadratic space $(\mathbb F_q^d,Q)$ admits a point set $P$ and a family $\mathcal S$ of distinct spheres of one fixed nonzero radius such that
\begin{equation}\label{eq:fixed-moderate-window}
 \frac{4q}{K_*^2}\le|\mathcal S|
 \le C_1K_*q^{(d-1)/2},
\end{equation}
while
\begin{equation}\label{eq:fixed-no-hyperplane}
 \max_H\frac{|P\cap H|}{|P|}\longrightarrow0.
\end{equation}
The number of exact centered blocks tends to infinity.
\item Let $d=3$, fix an integer $t\ge3$, and assume $C_1>t/\sqrt2$.  For all sufficiently large odd $q$, every nondegenerate ternary quadratic space admits a system with exactly $t$ blocks, one fixed nonzero radius, and
\begin{equation}\label{eq:d3-family-only-window}
 \frac{4q}{K_*^2}\le|\mathcal S|=tq
 \le C_1K_*q,
 \qquad
 K_*=(1-q^{-1})\sqrt2.
\end{equation}
Each diagonal block carries exactly $1/t$ of the total surplus and all cross-block centered discrepancies vanish.
\end{enumerate}
Consequently, the family-only lower range is not free of obstructions.  In $d\ge4$, no lower bound
\[
 |P\cap H|\ge cK_*^{-C}|P|
\]
with fixed $c>0$ and $C<\infty$ can follow from fixed-scale surplus and the moderate window, even for one nonzero radius.  For every fixed $\delta>0$, the examples in part~\textup{(1)} also satisfy $|P\cap L|\le\delta|P|$ for every proper affine flat $L$, and hence for every affine totally singular flat, once $q$ is sufficiently large.
\end{corollary}

\begin{proof}
For part~\textup{(1)}, use \cref{thm:fixed-block} with
\[
 t=\left\lfloor\frac{C_1\sqrt q}{8}\right\rfloor.
\]
By \eqref{eq:fixed-K}, $K_*\to1$.  Thus $t\to\infty$, the lower bound in \eqref{eq:fixed-moderate-window} holds for large $q$, and
\[
 |\mathcal S|=tq\le\frac{C_1}{8}q^{3/2}
 \le C_1K_*q^{(d-1)/2}
\]
for all sufficiently large $q$.  Equation \eqref{eq:fixed-no-hyperplane} follows from \eqref{eq:fixed-hyperplane-ratio}.  Every proper affine flat is contained in an affine hyperplane.

For part~\textup{(2)}, choose $R$ so that $N_R=2q$, as noted above.  Then \eqref{eq:fixed-K} gives $K_*=(1-q^{-1})\sqrt2$.  The lower inequality in \eqref{eq:d3-family-only-window} is equivalent to
\[
 t\ge\frac{2}{(1-q^{-1})^2},
\]
which holds for every fixed $t\ge3$ once $q$ is sufficiently large.  The upper inequality is equivalent to $t\le C_1(1-q^{-1})\sqrt2$, which follows for large $q$ from $C_1>t/\sqrt2$.  The centered-block assertions are \eqref{eq:fixed-cross} and \eqref{eq:fixed-sigma}.
\end{proof}

\begin{table}[ht]
\centering
\small
\caption{Location of the explicit obstruction families relative to the point-independent lower condition.}
\label{tab:range-obstructions}
\begin{tabularx}{\textwidth}{@{}P{0.27\textwidth}P{0.16\textwidth}Y Y@{}}
\toprule
Construction & Dimension & $|\mathcal S|\ge4q/K^2$? & Structural behavior \\
\midrule
Fixed-radius blocks, $t\asymp\sqrt q$ & $d\ge4$ & Yes & Unbounded exact direct sum; no macroscopic hyperplane or pair section. \\
Fixed-radius blocks, fixed $t\ge3$ & $d=3$ & Yes, when $C_1>t/\sqrt2$ & Genuine finite centered block system; each block carries $1/t$ of the surplus. \\
Rational-normal-curve design & $d\ge3$ & No & Complete rich-pair graph, diffuse parameter geometry, quantitative block indecomposability. \\
Fixed-radius twisted cubic & $d=3$ & No & Same diffuse behavior at the exact point-sensitive lower boundary. \\
\bottomrule
\end{tabularx}
\end{table}

The block construction is an exact model for the second branch of an atomicity dichotomy.  It also rules out the first branch at every fixed positive scale.

\begin{corollary}[No surplus-heavy sphere-pair section]\label{cor:no-heavy-section}
Assume $d\ge4$.  There is a constant $B_d'>0$ such that, for all sufficiently large $q$ and every pair of distinct spheres $S,S'\in\mathcal S$,
\begin{equation}\label{eq:no-heavy-section}
 \frac{\sigma(P\cap S\cap S',\mathcal S)}
 {\sigma(P,\mathcal S)}
 \le\frac{2q^{d-2}}{tN_R}
 \le\frac{B_d'}{t}.
\end{equation}
On the other hand, the partitions in \eqref{eq:fixed-global-def} satisfy
\begin{equation}\label{eq:block-surplus-share}
 \sigma(A_z,\mathcal S)=\sigma(A_z,\mathcal S_z)=(q-1)N_R
 =\frac1t\sigma(P,\mathcal S),
 \qquad
 \sigma(A_b,\mathcal S_z)=0\quad(b\ne z).
\end{equation}
Thus when $t\to\infty$, no sphere-pair section carries a fixed positive fraction of the surplus, while the second alternative is realized by an exact centered-incidence block decomposition.
\end{corollary}

\begin{proof}
Equation \eqref{eq:fixed-degree} and $|\mathcal S|/q=t$ show that every point of $P$ has the same positive excess degree $q-1$.  Hence
\[
 \sigma(A,\mathcal S)=(q-1)|A|
 \qquad(A\subseteq P).
\]
By \cref{lem:radical}, every distinct sphere pair has at most $2q^{d-2}$ common points.  Dividing by \eqref{eq:fixed-sigma} gives the first inequality in \eqref{eq:no-heavy-section}, and \eqref{eq:NR-asymptotic} gives the second.  Equation \eqref{eq:block-surplus-share} follows from \eqref{eq:fixed-cross} and additivity over the sphere blocks.
\end{proof}

\subsection{An overwhelmingly popular direction still need not concentrate}

\begin{proposition}[Popular direction distributed over a parallel class]\label{prop:fixed-popular}
Let $E_{[v]}$ denote the coincidence energy contributed by ordered pairs of distinct spheres whose center difference has projective direction $[v]$.  In the construction of \cref{thm:fixed-block},
\begin{equation}\label{eq:fixed-direction-energy}
 E_{[v]}=tq(q-1)N_R,
\end{equation}
and
\begin{equation}\label{eq:fixed-direction-ratio}
 \frac{E_{[v]}}{E_{\mathrm{off}}(P,\mathcal S)}
 =\frac{q(q-1)}{(q+t-1)(q+t-2)}.
\end{equation}
Consequently, if $t=o(q)$, one direction carries $1-o(1)$ of the full off-diagonal energy, but that energy is divided among the $t$ parallel hyperplanes $H_z$, each meeting $P$ in exactly $N_R=|P|/t$ points.
\end{proposition}

\begin{proof}
Pairs with center difference in direction $[v]$ are exactly the ordered pairs of distinct members of one $\mathcal S_z$.  There are $tq(q-1)$ such pairs, and \eqref{eq:fixed-common-section} shows that each contributes $N_R$, proving \eqref{eq:fixed-direction-energy}.

Every point has degree $D=q+t-1$ by \eqref{eq:fixed-degree}.  Hence
\[
 E_{\mathrm{off}}(P,\mathcal S)
 =|P|D(D-1)
 =tN_R(q+t-1)(q+t-2).
\]
Division gives \eqref{eq:fixed-direction-ratio}.  The final assertion follows from \eqref{eq:fixed-Hz} and the disjointness of the blocks.
\end{proof}

\subsection{Diffuse rational-normal-curve systems in every dimension}\label{subsec:rnc}

Szab\'o recently used rational normal curves in physical point space to construct sets with uniformly small intersections with $Q$-quadrics \cite{Szabo}.  The construction below is dual in spirit but has a different incidence purpose: the curve lies in lifted sphere-parameter space, and physical points are chosen so that their incidence hyperplanes cut out prescribed root sets on that curve.

We next construct a diffuse obstruction in every dimension.  Assume $d\ge3$.  A Witt decomposition gives
\[
 \mathbb F_q^d=\mathbb F_q e\oplus\mathbb F_q f\oplus V_0,
 \qquad
 Q(xe+yf+z)=xy+Q_\perp(z),
\]
where $Q_\perp$ is nondegenerate on the $(d-2)$-dimensional space $V_0$.  Fix a basis $u_1,\ldots,u_{d-2}$ of $V_0$ and put
\[
 z_t=\sum_{j=1}^{d-2}t^j u_j,
 \qquad
 c_t=t^{d-1}f+z_t
 \qquad(t\in\mathbb F_q).
\]
Set
\[
 \gamma=-1-Q(c_1),
 \qquad
 g(t)=Q(c_t)+t^d+\gamma t^{d-1}.
\]
Since $g(0)=g(1)=0$, the map $g:\mathbb F_q\to\mathbb F_q$ is not injective and hence, on a finite set of the same cardinality as its codomain, is not surjective.  Choose
\[
 \kappa\notin-g(\mathbb F_q)
\]
and define
\begin{equation}\label{eq:rnc-spheres}
 b_t=t^d+\gamma t^{d-1}+\kappa,
 \qquad
 r_t=Q(c_t)+b_t,
 \qquad
 S_t=S_Q(c_t,r_t).
\end{equation}
Then every $r_t$ is nonzero.

Let $\mathcal M_\gamma$ be the set of monic polynomials
\[
 F(T)=T^d+a_{d-1}T^{d-1}+\cdots+a_1T+a_0
 \quad\text{with }a_{d-1}\ne\gamma.
\]

\begin{proposition}[Polynomial-root chart]\label{prop:rnc-chart}
Put
\[
 U=\{xe+yf+z:x\in\mathbb F_q^\times,\ y\in\mathbb F_q,\ z\in V_0\}.
\]
The map $\Psi:U\to\mathcal M_\gamma$ defined by
\begin{equation}\label{eq:rnc-chart}
 \Psi(xe+yf+z)(T)
 =T^d+(x+\gamma)T^{d-1}
  +\sum_{j=1}^{d-2}2\ip{z}{u_j}T^j
  +\kappa-Q(xe+yf+z)
\end{equation}
is a bijection.  Most importantly,
\begin{equation}\label{eq:rnc-root-incidence}
 p\in S_t\quad\Longleftrightarrow\quad\Psi(p)(t)=0.
\end{equation}
\end{proposition}

\begin{proof}
The linear map
\[
 V_0\longrightarrow\mathbb F_q^{d-2},
 \qquad
 z\longmapsto(2\ip{z}{u_1},\ldots,2\ip{z}{u_{d-2}})
\]
is an isomorphism by nondegeneracy of $Q_\perp$.  Given the coefficients of $F\in\mathcal M_\gamma$, first set $x=a_{d-1}-\gamma\ne0$, then recover the unique $z\in V_0$ from the coefficients $a_1,\ldots,a_{d-2}$, and finally set
\[
 y=\frac{\kappa-Q_\perp(z)-a_0}{x}.
\]
This gives the inverse of \eqref{eq:rnc-chart}.

For $p=xe+yf+z$, the sphere equation and \eqref{eq:rnc-spheres} give
\begin{align*}
 Q(p-c_t)=r_t
 &\Longleftrightarrow
 Q(p)-2\ip{p}{c_t}=b_t\\
 &\Longleftrightarrow
 t^d+(x+\gamma)t^{d-1}
 +\sum_{j=1}^{d-2}2\ip{z}{u_j}t^j
 +\kappa-Q(p)=0,
\end{align*}
which is \eqref{eq:rnc-root-incidence}.
\end{proof}

Let
\[
 \mathfrak A_{d,\gamma}=
 \left\{A\in\binom{\mathbb F_q}{d}:-\sum_{a\in A}a\ne\gamma\right\}.
\]
For $A\in\mathfrak A_{d,\gamma}$, let $p_A$ be the point corresponding under \cref{prop:rnc-chart} to the polynomial $\prod_{a\in A}(T-a)$, and put
\begin{equation}\label{eq:rnc-system}
 P_{\mathrm{rn}}=\{p_A:A\in\mathfrak A_{d,\gamma}\},
 \qquad
 \mathcal S_{\mathrm{rn}}=\{S_t:t\in\mathbb F_q\}.
\end{equation}

\begin{theorem}[Diffuse rational-normal-curve sphere systems]\label{thm:rnc-system}
Fix $d\ge3$.  For all sufficiently large odd $q$, the system in \eqref{eq:rnc-system} has the following properties.  Write $N=|P_{\mathrm{rn}}|$ and $\Sigma=\sigma(P_{\mathrm{rn}},\mathcal S_{\mathrm{rn}})$.
\begin{enumerate}[label=\textup{(\arabic*)}]
\item Every sphere has nonzero radius.  The lifted parameters $\Phi(S_t)$ are affinely equivalent to
\[
 t\longmapsto(t,t^2,\ldots,t^d)\in\mathbb F_q^d.
\]
For every $0\le k\le d-1$, no affine $k$-flat in their $d$-dimensional affine span contains more than $k+1$ lifted parameters.  In particular, no coaxal pencil contains three members of $\mathcal S_{\mathrm{rn}}$.
\item The map $A\mapsto p_A$ is injective,
\begin{equation}\label{eq:rnc-N-bounds}
 \binom qd-\frac1d\binom q{d-1}
 \le N\le\binom qd,
 \qquad
 N=\frac{q^d}{d!}+O_d(q^{d-1}),
\end{equation}
and
\begin{equation}\label{eq:rnc-incidence}
 p_A\in S_t\quad\Longleftrightarrow\quad t\in A.
\end{equation}
Consequently every point has degree $d$.  For distinct $t,u$,
\begin{equation}\label{eq:rnc-pair-count}
 \binom{q-2}{d-2}-\frac1{d-2}\binom{q-2}{d-3}
 \le |P_{\mathrm{rn}}\cap S_t\cap S_u|
 \le\binom{q-2}{d-2}.
\end{equation}
\item One has
\begin{equation}\label{eq:rnc-surplus}
 I(P_{\mathrm{rn}},\mathcal S_{\mathrm{rn}})=dN,
 \qquad
 \Sigma=(d-1)N,
 \qquad
 K_q=\frac{(d-1)\sqrt N}{q^{d/2}}
 \longrightarrow\frac{d-1}{\sqrt{d!}}.
\end{equation}
Moreover,
\begin{equation}\label{eq:rnc-lower-exact}
 K_q^2q^{d-1}|\mathcal S_{\mathrm{rn}}|=(d-1)^2N\ge4N.
\end{equation}
For all sufficiently large $q$, the system lies in the moderate window of \cref{cor:moderate} with $C_1=2$, and its rich-pair graph is complete.
\item Every $B\subseteq P_{\mathrm{rn}}$ satisfies
\begin{equation}\label{eq:rnc-uniform-excess}
 \sigma(B,\mathcal S_{\mathrm{rn}})=(d-1)|B|.
\end{equation}
Hence, uniformly over distinct $t,u$,
\begin{equation}\label{eq:rnc-no-atom}
 \frac{\sigma(P_{\mathrm{rn}}\cap S_t\cap S_u,\mathcal S_{\mathrm{rn}})}{\Sigma}
 =O_d(q^{-2}).
\end{equation}
\end{enumerate}
\end{theorem}

\begin{proof}
The radii are nonzero by the choice of $\kappa$.  Relative to the basis $u_1,\ldots,u_{d-2},f$ and the last lift coordinate,
\[
 \Phi(S_t)=(2t,2t^2,\ldots,2t^{d-2},2t^{d-1},t^d+\gamma t^{d-1}+\kappa).
\]
An invertible affine change of coordinates sends this to $(t,t^2,\ldots,t^d)$.  If $t_0,\ldots,t_{k+1}$ are distinct, the augmented vectors
\[
 (1,t_i,t_i^2,\ldots,t_i^d)
\]
have rank $k+2$ by the Vandermonde determinant.  Thus those $k+2$ points are affinely independent, proving part~(1).

Part~(2) follows from the root chart.  To count the excluded sets, observe that every $d$-set with prescribed sum gives $d$ distinct $(d-1)$-subsets by deleting one element, while each $(d-1)$-subset has at most one completion to the prescribed sum.  Hence at most $d^{-1}\binom q{d-1}$ sets are excluded, proving \eqref{eq:rnc-N-bounds}.  The same deletion argument with two fixed labels shows that at most $(d-2)^{-1}\binom{q-2}{d-3}$ of the $\binom{q-2}{d-2}$ candidate $d$-sets in a pair section are excluded.  This proves \eqref{eq:rnc-pair-count}.

Since $|\mathcal S_{\mathrm{rn}}|=q$ and every point has degree $d$, the identities in \eqref{eq:rnc-surplus} are immediate.  Equation \eqref{eq:rnc-lower-exact} follows by substitution and uses $d\ge3$.  The moderate upper bound with $C_1=2$ follows from the positive limit of $K_q$ when $d=3$, and from $q=o(q^{(d-1)/2})$ when $d\ge4$.  The lower bound in \eqref{eq:rnc-pair-count} is asymptotic to $q^{d-2}/(d-2)!$; it therefore exceeds $K_qq^{(d-1)/2}/8$ for large $q$, including $d=3$.  Thus every pair is rich at the threshold of \cref{cor:moderate}.

Finally the baseline point degree is $|\mathcal S_{\mathrm{rn}}|/q=1$, so every point has excess degree $d-1$, proving \eqref{eq:rnc-uniform-excess}.  Divide the upper bound in \eqref{eq:rnc-pair-count} by \eqref{eq:rnc-N-bounds} to obtain \eqref{eq:rnc-no-atom}.
\end{proof}

\begin{remark}[Point-sensitive versus family-only lower ranges]\label{rem:rnc-lower-range}
The systems in \cref{thm:rnc-system} satisfy the exact point-sensitive lower condition
\[
 K_q^2q^{d-1}|\mathcal S_{\mathrm{rn}}|=(d-1)^2|P_{\mathrm{rn}}|,
\]
and in dimension three this is equality in the threshold $K^2q^{d-1}|\mathcal S|\ge4|P|$.  They do not satisfy the stronger point-independent sufficient condition $|\mathcal S|\ge4q/K_q^2$: indeed $|\mathcal S|=q$ whereas $4q/K_q^2>q$ for all sufficiently large $q$.  Accordingly, the diffuse obstruction disproves a two-branch label-coherence theorem under the point-sensitive hypotheses of \cref{cor:moderate}, but it does not settle whether a diffuse indecomposable branch exists in the denser family-only subrange.  Exact block obstructions do occur there by \cref{cor:fixed-no-proportional}.
\end{remark}

\begin{definition}[Root-evaluation incidence system]\label{def:root-system}
For $d\ge3$ and $\gamma\in\mathbb F_q$, let $\mathcal D_{d,\gamma}$ be the bipartite incidence system whose object set is $\mathbb F_q$, whose point set is
\[
 \mathfrak A_{d,\gamma}
 =\left\{A\in\binom{\mathbb F_q}{d}:-\sum_{a\in A}a\ne\gamma\right\},
\]
and in which $A$ is incident to $t$ exactly when $t\in A$.  Its centered discrepancy is defined using the baseline $1/q$.  The sphere system of \cref{thm:rnc-system} is incidence-isomorphic to $\mathcal D_{d,\gamma}$ for the value of $\gamma$ chosen in its construction.
\end{definition}

\begin{theorem}[Hypergraph universality on one lifted rational normal curve]\label{thm:hypergraph-universality}
Fix $d\ge3$ and use the sphere family $\mathcal S_{\mathrm{rn}}=\{S_t:t\in\mathbb F_q\}$ and points $p_A$ from \cref{prop:rnc-chart}.  For every $d$-uniform hypergraph
\[
 \mathcal G\subseteq\mathfrak A_{d,\gamma}
\]
on the label set $\mathbb F_q$, put
\[
 P_{\mathcal G}=\{p_A:A\in\mathcal G\}.
\]
Then
\begin{equation}\label{eq:universality-incidence}
 p_A\in S_t\quad\Longleftrightarrow\quad t\in A,
\end{equation}
so the point--sphere incidence graph of $(P_{\mathcal G},\mathcal S_{\mathrm{rn}})$ is exactly the membership graph of $\mathcal G$.  If $m=|\mathcal G|$, then
\begin{equation}\label{eq:universality-parameters}
 I(P_{\mathcal G},\mathcal S_{\mathrm{rn}})=dm,
 \qquad
 \sigma(P_{\mathcal G},\mathcal S_{\mathrm{rn}})=(d-1)m,
 \qquad
 K_{\mathcal G}=\frac{(d-1)\sqrt m}{q^{d/2}}.
\end{equation}
In particular,
\begin{equation}\label{eq:universality-lower}
 K_{\mathcal G}^2q^{d-1}|\mathcal S_{\mathrm{rn}}|
 =(d-1)^2m\ge4m,
\end{equation}
so every nonempty member of this universal class satisfies the point-sensitive diagonal-dominance condition of the scale-adaptive inverse theorem.

For each fixed $\rho>0$, all systems with $m\ge\rho|\mathfrak A_{d,\gamma}|$ have $K_{\mathcal G}\asymp_{d,\rho}1$.  For sufficiently large $q$ they lie in the moderate range with $C_1=1$ when $d\ge4$, and with every fixed
\[
 C_1>\sqrt{\frac{3}{2\rho}}
\]
when $d=3$.
\end{theorem}

\begin{proof}
Equation \eqref{eq:universality-incidence} is \eqref{eq:rnc-root-incidence} restricted to the chosen set of split polynomials.  Every point has degree exactly $d$ and $|\mathcal S_{\mathrm{rn}}|=q$, so the baseline contribution is $m$ and \eqref{eq:universality-parameters} follows.  Equation \eqref{eq:universality-lower} is immediate and uses $d\ge3$.

By \eqref{eq:rnc-N-bounds},
\[
 |\mathfrak A_{d,\gamma}|=\frac{q^d}{d!}+O_d(q^{d-1}).
\]
Thus $m\ge\rho|\mathfrak A_{d,\gamma}|$ gives
\[
 K_{\mathcal G}\ge(d-1)\sqrt{\frac\rho{d!}}+o_{d,\rho}(1).
\]
The moderate upper condition is $q\le C_1K_{\mathcal G}q^{(d-1)/2}$.  It is automatic for large $q$ when $d\ge4$, and the displayed condition on $C_1$ makes it hold when $d=3$.
\end{proof}

The universality theorem gives a precise obstruction to a global finite-model Freiman theorem at the fixed-scale threshold.  It is stronger than exhibiting several exceptional configurations: even with the sphere family and its lifted rational normal curve fixed once and for all, the point side can encode arbitrary dense uniform hypergraphs.

\begin{corollary}[Exponential metric entropy for individual templates]\label{cor:no-freiman-list}
Fix $d\ge3$ and $0<\varepsilon<1/2$, and put
\[
 N_q=|\mathfrak A_{d,\gamma}|.
\]
Let $\mathfrak T_q$ be any family of point sets in $\mathbb F_q^d$ satisfying
\[
 \log|\mathfrak T_q|=o(N_q).
\]
For all sufficiently large $q$, there is a hypergraph $\mathcal G\subseteq\mathfrak A_{d,\gamma}$ with
\begin{equation}\label{eq:middle-hypergraph}
 \frac{N_q}{3}\le|\mathcal G|\le\frac{2N_q}{3}
\end{equation}
and
\begin{equation}\label{eq:no-template-close}
 |P_{\mathcal G}\triangle T|>\varepsilon N_q
 \qquad(T\in\mathfrak T_q).
\end{equation}
Every resulting point--sphere system has a fixed-scale parameter bounded above and below by positive constants depending only on $d$, satisfies the point-sensitive lower condition, and lies in the moderate range for large $q$ (with $C_1=1$ if $d\ge4$, and with any $C_1>3/\sqrt2$ if $d=3$).

Consequently, any classification of all constant-$K$ fixed-scale systems up to $o(q^d)$ point edits requires $\exp(\Omega_d(q^d))$ distinct templates, even when the sphere family and its lifted rational normal curve are fixed.  In particular, no finite, polynomial-size, or more generally subexponential list of individual point-set templates can classify the full fixed-scale class in this edit metric.
\end{corollary}

\begin{proof}
Replacing each template $T$ by $T\cap P_{\mathrm{rn}}$ can only decrease its symmetric-difference distance from every $P_{\mathcal G}$.  The map $\mathcal G\mapsto P_{\mathcal G}$ is an isometry between symmetric-difference metrics on subsets of $\mathfrak A_{d,\gamma}$ and subsets of $P_{\mathrm{rn}}$.  The number of subsets satisfying \eqref{eq:middle-hypergraph} is
\[
 2^{N_q}-2\sum_{j<N_q/3}\binom{N_q}{j}
 =\exp\bigl((\log2-o(1))N_q\bigr).
\]
On the other hand, a Hamming ball of radius $\varepsilon N_q$ contains at most
\[
 \sum_{j\le\varepsilon N_q}\binom{N_q}{j}
 =\exp\bigl((h(\varepsilon)+o(1))N_q\bigr)
\]
subsets, where
\[
 h(\varepsilon)=-\varepsilon\log\varepsilon-(1-\varepsilon)\log(1-\varepsilon)<\log2.
\]
Since $\log|\mathfrak T_q|=o(N_q)$, the union of the corresponding balls cannot cover all middle-sized subsets for large $q$.  This proves \eqref{eq:no-template-close}.  The quantitative incidence assertions follow from \cref{thm:hypergraph-universality} with $\rho=1/3$.
\end{proof}

\begin{definition}[Quantitatively nontrivial centered blocks]\label{def:centered-blocks}
Let $\Sigma=\sigma(P,\mathcal S)>0$, let $0<\tau\le1$, and let $\varepsilon\ge0$.  A $(\tau,\varepsilon)$-centered block decomposition consists of an integer $m\ge2$ and matched partitions
\[
 \mathcal S=\mathcal S_1\sqcup\cdots\sqcup\mathcal S_m,
 \qquad
 P=A_1\sqcup\cdots\sqcup A_m\sqcup R,
\]
such that
\begin{align}
 \sigma(A_i,\mathcal S_i)&\ge\tau\Sigma
 &&(1\le i\le m),\label{eq:block-diagonal}\\
 \sigma(R,\mathcal S)&\le\varepsilon\Sigma,\label{eq:block-remainder}\\
 \sum_{i\ne j}\sigma(A_i,\mathcal S_j)_+&\le\varepsilon\Sigma,
 \qquad x_+=\max\{x,0\}.\label{eq:block-cross}
\end{align}
The diagonal lower bound prevents vacuous refinement into arbitrarily small blocks.  Controlling only positive cross discrepancy is weaker than controlling the sum of absolute cross discrepancies.
\end{definition}

\begin{theorem}[Quantitative centered-block indecomposability]\label{thm:rnc-block-indecomp}
Fix $d\ge3$ and put
\[
 c_d=\frac{d-1}{32d^2},
 \qquad
 C_d=8d(d+1).
\]
Uniformly for every $\gamma\in\mathbb F_q$, every $0<\tau\le1$, and every odd $q\ge C_d/\tau$, the root-evaluation system $\mathcal D_{d,\gamma}$ admits no $(\tau,\varepsilon)$-centered block decomposition with
\[
 \varepsilon\le c_d\tau^2.
\]
In particular, this conclusion holds for the rational-normal-curve sphere system of \cref{thm:rnc-system}.  More quantitatively, writing $P$ and $\mathcal S$ for the two sides of $\mathcal D_{d,\gamma}$ and $\Sigma=\sigma(P,\mathcal S)$, every matched partition satisfying \eqref{eq:block-diagonal} and
\[
 \sigma(R,\mathcal S)\le c_d\tau^2\Sigma
\]
obeys
\begin{equation}\label{eq:rnc-cross-lower}
 \sum_{i\ne j}\sigma(A_i,\mathcal S_j)_+
 \ge2c_d\tau^2\Sigma.
\end{equation}
\end{theorem}

\begin{proof}
Identify each object with its label $t\in\mathbb F_q$ and each point with its $d$-element incidence neighborhood.  Put $N=|P|$.  The deletion argument used in the proof of \eqref{eq:rnc-N-bounds} is uniform in $\gamma$, and gives
\[
 \binom qd-\frac1d\binom q{d-1}\le N\le\binom qd.
\]
Every point has degree $d$, $|\mathcal S|=q$, and therefore $\Sigma=(d-1)N$.  Write
\[
 s_i=|\mathcal S_i|,
 \qquad
 \theta_i=\frac{s_i}{q},
 \qquad
 n_i=|A_i|,
\]
and let $n_i^0$ be the number of points of $A_i$ whose entire incidence neighborhood lies in $\mathcal S_i$.  Put
\[
 \mu=\frac{d!N}{q^d},
 \qquad
 \rho=\frac{\tau(d-1)}{2d}.
\]
The lower bound for $N$ gives the explicit estimate
\begin{align}
 \mu
 &\ge \prod_{j=0}^{d-1}\left(1-\frac jq\right)-\frac1q\notag\\
 &\ge1-\frac{d(d-1)/2+1}{q}
 \ge1-\frac\rho4,
 \label{eq:rnc-mu}
\end{align}
where the last inequality follows from $q\ge C_d/\tau$ and $d(d-1)/2+1\le(d+1)(d-1)$.

A fixed object belongs to at most $\binom{q-1}{d-1}\le q^{d-1}/(d-1)!$ points.  If $D_i=\sigma(A_i,\mathcal S_i)$, then
\[
 D_i\le I(P,\mathcal S_i)
 \le\frac{d\theta_iq^d}{d!}.
\]
On the other hand, \eqref{eq:block-diagonal} and $\Sigma=(d-1)N$ give
\[
 D_i\ge\tau(d-1)N
 =\frac{\tau(d-1)\mu q^d}{d!}.
\]
Thus $\theta_i\ge2\rho\mu\ge\rho$.  Since $m\ge2$, also $\theta_i\le1-\rho$ for every $i$.

For a point $e\in A_i$, let $k_i(e)$ be the number of its $d$ incident object labels lying in $\mathcal S_i$.  Summing the signed cross discrepancies leaving the point block $A_i$ gives
\[
 C_i:=\sum_{j\ne i}\sigma(A_i,\mathcal S_j)
 =\sum_{e\in A_i}\bigl((d-k_i(e))-(1-\theta_i)\bigr).
\]
A point internal to $\mathcal S_i$ contributes $-(1-\theta_i)$, whereas every other point contributes at least $\theta_i$.  Therefore
\begin{equation}\label{eq:rnc-row-cross}
 C_i\ge\theta_i n_i-n_i^0.
\end{equation}
Put
\[
 X=\sum_{i\ne j}\sigma(A_i,\mathcal S_j)_+.
\]
Using \eqref{eq:rnc-row-cross}, $\theta_i\ge\rho$, and the elementary inequality $\sum_j(x_j)_+\ge(\sum_jx_j)_+$, we obtain
\begin{align}
 X
 &\ge\sum_i(C_i)_+
 \ge\sum_i(\theta_i n_i-n_i^0)_+\notag\\
 &\ge\rho\sum_i\left(n_i-\frac{n_i^0}{\theta_i}\right)_+
 \ge\rho\left(\sum_i n_i-\sum_i\frac{n_i^0}{\theta_i}\right)_+.
 \label{eq:rnc-X}
\end{align}

Now
\[
 n_i^0\le\binom{s_i}{d}\le\frac{\theta_i^dq^d}{d!},
\]
so
\begin{align*}
 \sum_i\frac{n_i^0}{\theta_i}
 &\le\frac{q^d}{d!}\sum_i\theta_i^{d-1}
 \le\frac{q^d}{d!}(1-\rho)^{d-2}\\
 &\le\frac{q^d}{d!}(1-\rho)
 \le\left(1-\frac\rho2\right)N.
\end{align*}
The last inequality follows from $\mu\ge1-\rho/4$ and
\[
 \frac{1-\rho}{1-\rho/4}\le1-\frac\rho2.
\]
Since every point has full-family excess degree $d-1$, the remainder hypothesis implies
\[
 |R|\le c_d\tau^2N\le\frac\rho4N;
\]
the last inequality holds because $0<\tau\le1$ and $d\ge3$.
Thus $\sum_i n_i=N-|R|\ge(1-\rho/4)N$.  Substitution in \eqref{eq:rnc-X} yields
\[
 X\ge\frac{\rho^2}{4}N
 =\frac{\tau^2(d-1)}{16d^2}\Sigma
 =2c_d\tau^2\Sigma,
\]
which proves \eqref{eq:rnc-cross-lower} and the theorem.
\end{proof}

\subsection{A fixed-radius twisted-cubic obstruction in dimension three}\label{subsec:fixed-tc}

The diffuse phenomenon persists after fixing one prescribed nonzero radius in dimension three.  Assume $\operatorname{char}\mathbb F_q\ne3$ (the standing hypotheses already exclude characteristic two) and put
\[
 Q_0(x,y,z)=xy+z^2.
\]
Fix $R\in\mathbb F_q^\times$ and define
\begin{equation}\label{eq:tc-spheres}
 c_t=(0,t^3,t),
 \qquad
 S_t=S_{Q_0}(c_t,R)
 \qquad(t\in\mathbb F_q).
\end{equation}
Then
\begin{equation}\label{eq:tc-lift}
 \Phi(S_t)=(0,2t^3,2t,R-t^2),
\end{equation}
which is an affine twisted cubic in its three-dimensional affine span.  The finite-projective geometry of the twisted cubic is classical; see, for example, \cite{BDMP}.

Let $\mathcal C^\times$ be the set of monic cubics
\[
 F(T)=T^3-s_1T^2+s_2T-s_3
 \qquad(s_1\ne0).
\]

\begin{proposition}[Fixed-radius cubic-root chart]\label{prop:tc-chart}
The map $\Psi:\{(x,y,z)\in\mathbb F_q^3:x\ne0\}\to\mathcal C^\times$ defined by
\begin{equation}\label{eq:tc-chart}
 \Psi(x,y,z)(T)
 =T^3-\frac1xT^2+\frac{2z}{x}T-\frac{xy+z^2-R}{x}
\end{equation}
is a bijection.  Its inverse sends $T^3-s_1T^2+s_2T-s_3$ to
\begin{equation}\label{eq:tc-chart-inverse}
 \left(\frac1{s_1},\ s_3+s_1R-\frac{s_2^2}{4s_1},\ \frac{s_2}{2s_1}\right).
\end{equation}
Moreover,
\begin{equation}\label{eq:tc-root-incidence}
 p\in S_t\quad\Longleftrightarrow\quad\Psi(p)(t)=0.
\end{equation}
\end{proposition}

\begin{proof}
For $p=(x,y,z)$, direct expansion gives
\[
 Q_0(p-c_t)-R=xy-xt^3+z^2-2zt+t^2-R.
\]
Multiplication by $-x^{-1}$ proves \eqref{eq:tc-root-incidence} and gives \eqref{eq:tc-chart}.  Reading off the coefficients gives \eqref{eq:tc-chart-inverse}; substitution verifies that the two maps are inverse.
\end{proof}

Let
\[
 \mathfrak T=\left\{A\in\binom{\mathbb F_q}{3}:\sum_{a\in A}a\ne0\right\}.
\]
For $A\in\mathfrak T$, let $p_A$ correspond under \cref{prop:tc-chart} to $\prod_{a\in A}(T-a)$, and put
\[
 P_{\mathrm{tc}}=\{p_A:A\in\mathfrak T\},
 \qquad
 \mathcal S_{\mathrm{tc}}=\{S_t:t\in\mathbb F_q\}.
\]

\begin{theorem}[Fixed-radius diffuse twisted-cubic system]\label{thm:fixed-tc}
Let $q\ge5$ and $\operatorname{char}\mathbb F_q\notin\{2,3\}$.  For the preceding system, put $N=|P_{\mathrm{tc}}|$ and $\Sigma=\sigma(P_{\mathrm{tc}},\mathcal S_{\mathrm{tc}})$.  Then:
\begin{enumerate}[label=\textup{(\arabic*)}]
\item Every sphere has the prescribed radius $R\ne0$, the lifted parameters are pairwise distinct, and no affine line contains three of them.  Hence no coaxal pencil contains three members of $\mathcal S_{\mathrm{tc}}$.
\item One has
\begin{equation}\label{eq:tc-N}
 N=\frac{(q-1)^2(q-2)}6,
 \qquad
 p_A\in S_t\Longleftrightarrow t\in A.
\end{equation}
Consequently every point has degree three and
\begin{equation}\label{eq:tc-surplus}
 I(P_{\mathrm{tc}},\mathcal S_{\mathrm{tc}})=3N,
 \qquad
 \Sigma=2N,
 \qquad
 K_q=\frac{2\sqrt N}{q^{3/2}}\longrightarrow\sqrt{\frac23}.
\end{equation}
In particular, $K_q^2q^2|\mathcal S_{\mathrm{tc}}|=4|P_{\mathrm{tc}}|$.
\item For distinct $a,b\in\mathbb F_q$,
\begin{equation}\label{eq:tc-pairs}
 q-3\le|P_{\mathrm{tc}}\cap S_a\cap S_b|\le q-2.
\end{equation}
The system lies in the moderate window of \cref{cor:moderate} with $C_1=2$, and its rich-pair graph is complete.
\item Every $B\subseteq P_{\mathrm{tc}}$ satisfies
\[
 \sigma(B,\mathcal S_{\mathrm{tc}})=2|B|,
\]
and every distinct sphere pair obeys
\begin{equation}\label{eq:tc-no-atom}
 \frac{\sigma(P_{\mathrm{tc}}\cap S_a\cap S_b,\mathcal S_{\mathrm{tc}})}{\Sigma}
 \le\frac6{(q-1)^2}.
\end{equation}
\item The incidence system is $\mathcal D_{3,0}$ from \cref{def:root-system}.  Hence, for every $0<\tau\le1$ and all $q\ge C_3/\tau$, it has no $(\tau,\varepsilon)$-centered block decomposition with $\varepsilon\le\tau^2/144$.
\end{enumerate}
Every nondegenerate ternary quadratic form in characteristic different from two and three admits an incidence-isomorphic example with one fixed nonzero radius.
\end{theorem}

\begin{proof}
The lift formula is \eqref{eq:tc-lift}.  Projection onto the coordinates $(2t,R-t^2)$ gives a nondegenerate parabola, so three distinct lifted parameters cannot be collinear.  The coaxal assertion follows from \cref{lem:coaxal}.

The root chart proves the incidence equivalence and injectivity of $A\mapsto p_A$.  There are $\binom q3$ three-element subsets in total.  The number with sum zero is $(q-1)(q-2)/6$.  Indeed, among the $q(q-1)$ ordered pairs $(a,b)$ with $a\ne b$, the completion $c=-a-b$ fails to be distinct precisely when $b=-2a$ or $a=-2b$; in characteristic different from two and three these are two disjoint families of $q-1$ pairs.  Thus there are $(q-1)(q-2)$ ordered triples of distinct elements with sum zero.  Subtraction gives \eqref{eq:tc-N}.  The identities in \eqref{eq:tc-surplus} follow because every point has degree three and $|\mathcal S_{\mathrm{tc}}|/q=1$.

For a fixed pair $a\ne b$, the third label may be any element outside $\{a,b\}$ except possibly $-a-b$, proving \eqref{eq:tc-pairs}.  Formula \eqref{eq:tc-N} gives $1/2\le K_q<1$ for $q\ge5$, so $|\mathcal S_{\mathrm{tc}}|=q\le2K_qq$, which is the moderate upper bound with $C_1=2$; the lower moderate condition is the exact identity following \eqref{eq:tc-surplus}.  Since $q-3\ge q/8>K_qq/8$, every pair is rich for $C_1=2$.

The full-family excess degree is two, proving the subset identity and \eqref{eq:tc-no-atom}.  Part~(5) follows from \cref{thm:rnc-block-indecomp}, because $c_3=1/144$.  Finally, every nondegenerate ternary form over such a field is similar to $Q_0$: isotropy splits off a hyperbolic plane, and a similarity adjusts the coefficient of the remaining one-dimensional summand.  A linear similarity preserves incidences, sends one nonzero radius to one nonzero radius, and acts invertibly on the lifted parameter space.
\end{proof}

\begin{remark}[Scope of the fixed-radius obstruction]\label{rem:tc-lower-range}
The system of \cref{thm:fixed-tc} satisfies $K_q^2q^2|\mathcal S_{\mathrm{tc}}|=4|P_{\mathrm{tc}}|$, so it lies exactly on the point-sensitive lower boundary in dimension three.  Since $|\mathcal S_{\mathrm{tc}}|=q$ and $K_q^2\to2/3$, it does not satisfy $|\mathcal S|\ge4q/K_q^2$.  Thus fixing the radius does not restore label coherence at the sharp point-sensitive threshold.  The family-only range already contains exact block systems by \cref{cor:fixed-no-proportional}; what this construction leaves open is a diffuse indecomposable example, or a theorem excluding one, in that denser range.
\end{remark}

\section{Affine hyperplanes and pinned configurations}\label{sec:applications}

\subsection{The centered affine-hyperplane design}

Let $\Omega_{\mathcal H}$ be the family of all affine hyperplanes in $\mathbb F_q^d$.  Its cardinality is
\[
 |\Omega_{\mathcal H}|=q\frac{q^d-1}{q-1}.
\]

\begin{lemma}[Affine-hyperplane design]\label{lem:hyperplane-design}
Let $A$ be the point--hyperplane incidence matrix and put $B=A-q^{-1}J_0$.  Then
\[
 BB^{\mathsf T}=q^{d-1}I-q^{-1}J,
 \qquad B\mathbf 1_{\Omega_{\mathcal H}}=0.
\]
Consequently, for every $X\subseteq\mathbb F_q^d$ and $\mathcal H\subseteq\Omega_{\mathcal H}$,
\begin{equation}\label{eq:hyperplane-complement}
 \left|I(X,\mathcal H)-\frac{|X||\mathcal H|}{q}\right|
 \le
 \sqrt{q^{d-1}|X|-\frac{|X|^2}{q}}
 \sqrt{|\mathcal H|\left(1-\frac{|\mathcal H|}{|\Omega_{\mathcal H}|}\right)}.
\end{equation}
Equality holds for a singleton $X$ and the family of all affine hyperplanes through its point.
\end{lemma}

\begin{proof}
Every point lies on $(q^d-1)/(q-1)$ affine hyperplanes, while two distinct points lie on $(q^{d-1}-1)/(q-1)$ common hyperplanes.  Hence
\[
 AA^{\mathsf T}=\frac{q^{d-1}-1}{q-1}J+q^{d-1}I.
\]
Since every row of $A$ has sum $(q^d-1)/(q-1)$ and $|\Omega_{\mathcal H}|=q(q^d-1)/(q-1)$, expanding the centered product gives the matrix identity.  The proof of \eqref{eq:hyperplane-complement} is then identical to the proof of \cref{prop:sphere-design}, and the point-star example gives equality.
\end{proof}

Dropping the complement factor gives Vinh's point--hyperplane estimate \cite{Vinh} with its exact centered constant.

\subsection{A scale-adaptive point--hyperplane inverse theorem}

\begin{theorem}[Rich codimension-two sections]\label{thm:hyperplane-adaptive}
Let $q$ be any prime power, let $d\ge2$, let $X\subseteq\mathbb F_q^d$ be nonempty, and let $\mathcal H$ be a nonempty family of distinct affine hyperplanes.  Put
\[
 n=|X|,\qquad h=|\mathcal H|,\qquad I_H=I(X,\mathcal H),
 \qquad \mathcal E_H=\frac{I_H^2}{n}-I_H.
\]
If $\mathcal E_H>0$, then for every $0<\eta<1$, ordered pairs satisfying
\[
 |X\cap H\cap H'|\ge\eta\frac{\mathcal E_H}{h^2}
\]
carry at least $(1-\eta)\mathcal E_H$ units of coincidence energy.  In particular, there are at least
\[
 \frac{(1-\eta)\mathcal E_H}{q^{d-2}}
\]
ordered such pairs, all nonparallel.

If, more specifically,
\[
 I(X,\mathcal H)-\frac{nh}{q}
 \ge Kq^{(d-1)/2}\sqrt{nh},
 \qquad K^2q^{d-1}h\ge4n,
\]
then at least $K^2qh/4$ ordered nonparallel pairs satisfy
\begin{equation}\label{eq:hyperplane-adaptive-section}
 |X\cap H\cap H'|\ge\frac{K^2q^{d-1}}{4h}.
\end{equation}
Under the additional cap $h\le C_1Kq^{(d-1)/2}$, where $C_1\ge1$, the right-hand side in \eqref{eq:hyperplane-adaptive-section} is at least $Kq^{(d-1)/2}/(4C_1)$.
\end{theorem}

\begin{proof}
The proof of \cref{thm:exact-adaptive} applies verbatim to the point--hyperplane incidence relation.  Distinct parallel hyperplanes are disjoint, while a nonparallel pair meets in a codimension-two flat of size $q^{d-2}$.  The fixed-scale consequences follow from the same estimate
\[
 \mathcal E_H\ge\frac12K^2q^{d-1}h
\]
used in \cref{cor:fixed-adaptive}, followed by the choice $\eta=1/2$.
\end{proof}

\subsection{Pin-generated systems and deficiency-sensitive pin counts}

Pinned distances and dot products are classical applications of finite-field incidence estimates; see, for example, \cite{Chapman,HIKR}.  For $p\in\mathbb F_q^d$ and nonempty sets $P,Y\subseteq\mathbb F_q^d$, define
\[
 \Delta_p^Q(P)=\{Q(x-p):x\in P\},
 \qquad
 \Pi_p(Y)=\{\langle p,x\rangle:x\in Y\},
\]
where the dot product is taken with respect to any fixed nondegenerate bilinear pairing.

Let $T$ be a set of $m$ distance pins and put
\[
 \mathcal S_T=\{S_Q(p,t):p\in T,\ t\in\Delta_p^Q(P)\},
 \qquad M_\Delta=|\mathcal S_T|.
\]
Every $x\in P$ lies on exactly one member of $\mathcal S_T$ for each pin, so
\begin{equation}\label{eq:pin-sphere-system}
 I(P,\mathcal S_T)=m|P|,
 \qquad
 \sigma(P,\mathcal S_T)=\beta_\Delta m|P|,
 \qquad
 \beta_\Delta=1-\frac{M_\Delta}{mq}.
\end{equation}
The exact fixed-scale parameter is
\begin{equation}\label{eq:Kdelta}
 K_\Delta
 =\frac{\beta_\Delta}{\sqrt{1-\beta_\Delta}}
 \frac{\sqrt{m|P|}}{q^{d/2}}.
\end{equation}

For dot products, let $T$ be a set of pairwise nonproportional nonzero pins and define
\[
 H_{p,t}=\{x:\langle p,x\rangle=t\},
 \qquad
 \mathcal H_T=\{H_{p,t}:p\in T,\ t\in\Pi_p(Y)\}.
\]
All these hyperplanes are distinct.  With $M_\Pi=|\mathcal H_T|$ and $\beta_\Pi=1-M_\Pi/(mq)$, the analogues of \eqref{eq:pin-sphere-system} and \eqref{eq:Kdelta} hold with $(P,\Delta)$ replaced by $(Y,\Pi)$.

\begin{proposition}[Deficiency-sensitive pin-count bounds]\label{prop:pin-counts}
Let $d\ge2$ and $0<\alpha\le1$.
\begin{enumerate}[label=\textup{(\arabic*)}]
\item Let $q$ be odd and let $P\subseteq\mathbb F_q^d$ be nonempty.  The number of pins $p\in\mathbb F_q^d$ satisfying
\[
 |\Delta_p^Q(P)|\le(1-\alpha)q
\]
is at most
\begin{equation}\label{eq:distance-pin-count}
 \frac{(1-\alpha)(q^{d+1}-q^d)}{\alpha^2|P|}.
\end{equation}
\item Let $q$ be any prime power, let $Y\subseteq\mathbb F_q^d$ be nonempty, and fix a nondegenerate bilinear pairing on $\mathbb F_q^d$.  Every set of pairwise nonproportional nonzero pins satisfying
\[
 |\Pi_p(Y)|\le(1-\alpha)q
\]
has cardinality at most
\begin{equation}\label{eq:dot-pin-count}
 \frac{(1-\alpha)(q^d-|Y|)}{\alpha^2|Y|}.
\end{equation}
\end{enumerate}
\end{proposition}

\begin{proof}
For distance pins, let $T$ be the deficient set and $m=|T|$.  If $m=0$ there is nothing to prove, so assume $m>0$.  Then $M_\Delta\le(1-\alpha)mq$ and $\sigma(P,\mathcal S_T)\ge\alpha m|P|$.  Squaring \eqref{eq:complement-sphere}, dropping only its final complement factor, and using the bound on $M_\Delta$ gives
\[
 \alpha^2m^2|P|^2
 \le(q^d-q^{d-1})|P|(1-\alpha)mq,
\]
which is \eqref{eq:distance-pin-count}.

For dot products, the claim is trivial when the chosen pin set is empty; otherwise use \eqref{eq:hyperplane-complement}.  Since $M_\Pi\le(1-\alpha)mq$,
\[
 \alpha^2m^2|Y|^2
 \le\left(q^{d-1}|Y|-\frac{|Y|^2}{q}\right)(1-\alpha)mq.
\]
After cancellation this is \eqref{eq:dot-pin-count}.
\end{proof}

\subsection{Deficient pinned distances}

\begin{theorem}[Deficient pinned distances force rich common sections]\label{thm:pinned-distance}
Let $q$ be odd, let $d\ge2$, let $P\subseteq\mathbb F_q^d$ be nonempty and let $T\subseteq\mathbb F_q^d$ be a set of $m\ge2$ pins.  Suppose
\[
 |\Delta_p^Q(P)|\le(1-\alpha)q
 \qquad(p\in T)
\]
for some $0<\alpha<1$, and suppose that for some $C_1\ge1$,
\begin{equation}\label{eq:pinned-upper-condition}
 (1-\alpha)^3mq^3
 \le C_1^2\alpha^2|P|.
\end{equation}
Then there exist distinct pins $p,p'\in T$ and attained radii $t\in\Delta_p^Q(P)$, $t'\in\Delta_{p'}^Q(P)$ for which
\begin{equation}\label{eq:pinned-distance-section}
 |P\cap S_Q(p,t)\cap S_Q(p',t')|
 \ge
 \frac{\alpha}{8C_1\sqrt{1-\alpha}}
 \sqrt{\frac{m|P|}{q}}.
\end{equation}
The same lower bound holds for the corresponding radical hyperplane.  Moreover, at least
\begin{equation}\label{eq:pinned-distance-count}
 \frac{3|P|m(m-1)}{4q^{d-1}}
\end{equation}
ordered choices $((p,t),(p',t'))$ satisfy \eqref{eq:pinned-distance-section}; if $d\ge3$, the lower bound improves to
\begin{equation}\label{eq:pinned-distance-count-d3}
 \frac{3|P|m(m-1)}{8q^{d-2}}.
\end{equation}
\end{theorem}

\begin{proof}
Use the notation in \eqref{eq:pin-sphere-system}.  Since $\beta_\Delta\ge\alpha$ and $(1-x)^3/x^2$ decreases on $(0,1)$, condition \eqref{eq:pinned-upper-condition} implies
\[
 M_\Delta\le C_1K_\Delta q^{(d-1)/2}.
\]
Put $\lambda=K_\Delta q^{(d-1)/2}/(8C_1)$.  Every point of $P$ has degree exactly $m$ in the generated sphere system.  Hence
\begin{equation}\label{eq:pin-exact-off}
 E_{\mathrm{off}}(P,\mathcal S_T)=|P|m(m-1).
\end{equation}
Pairs contributing less than $\lambda$ carry at most
\[
 \lambda M_\Delta^2
 \le\frac18K_\Delta^2q^{d-1}M_\Delta
 =\frac18\beta_\Delta^2m^2|P|
 \le\frac14|P|m(m-1),
\]
where the final inequality uses $m\ge2$.  Thus rich pairs carry at least three quarters of \eqref{eq:pin-exact-off}.  A contributing pair cannot come from the same pin, because distinct concentric spheres are disjoint.  The size bound \eqref{eq:pinned-distance-section} follows from \eqref{eq:Kdelta} and the monotonicity of $x/\sqrt{1-x}$.  Dividing the rich energy by $q^{d-1}$ gives \eqref{eq:pinned-distance-count}; for $d\ge3$, divide by $2q^{d-2}$ and use \cref{lem:radical}.
\end{proof}

\subsection{Deficient pinned dot products}

\begin{theorem}[Deficient pinned dot products force rich codimension-two sections]\label{thm:pinned-dot}
Let $q$ be any prime power, let $d\ge2$, let $Y\subseteq\mathbb F_q^d$ be nonempty, and let $T$ be a set of $m\ge2$ pairwise nonproportional nonzero pins.  Suppose
\[
 |\Pi_p(Y)|\le(1-\alpha)q
 \qquad(p\in T)
\]
for some $0<\alpha<1$ and $C_1\ge1$, and suppose
\begin{equation}\label{eq:pinned-dot-upper}
 (1-\alpha)^3mq^3
 \le C_1^2\alpha^2|Y|.
\end{equation}
Then there exist distinct pins $p,p'\in T$ and attained values $t\in\Pi_p(Y)$, $t'\in\Pi_{p'}(Y)$ for which
\begin{equation}\label{eq:pinned-dot-section}
 |\{x\in Y:\langle p,x\rangle=t,\ \langle p',x\rangle=t'\}|
 \ge
 \frac{\alpha}{8C_1\sqrt{1-\alpha}}
 \sqrt{\frac{m|Y|}{q}}.
\end{equation}
At least
\begin{equation}\label{eq:pinned-dot-count}
 \frac{3|Y|m(m-1)}{4q^{d-2}}
\end{equation}
ordered choices $((p,t),(p',t'))$ satisfy \eqref{eq:pinned-dot-section}.
\end{theorem}

\begin{proof}
The generated hyperplane family has exact point degree $m$, so its off-diagonal energy is $|Y|m(m-1)$.  With
\[
 K_\Pi=\frac{\beta_\Pi}{\sqrt{1-\beta_\Pi}}
 \frac{\sqrt{m|Y|}}{q^{d/2}},
\]
condition \eqref{eq:pinned-dot-upper} gives $M_\Pi\le C_1K_\Pi q^{(d-1)/2}$.  Put $\lambda=K_\Pi q^{(d-1)/2}/(8C_1)$.  Exactly as in the proof of \cref{thm:pinned-distance}, pairs below $\lambda$ carry at most one quarter of the off-diagonal energy.  Every contributing pair comes from distinct pins and is nonparallel because the pins are nonproportional.  Dividing the remaining energy by $q^{d-2}$ gives \eqref{eq:pinned-dot-count}, and the threshold is at least \eqref{eq:pinned-dot-section}.
\end{proof}

\section{Discussion: a classification theorem and its sharp boundary}\label{sec:discussion}

There are two mathematically different inverse problems in this paper.  At the maximal-incidence endpoint, the problem has a genuine Freiman-type answer.  \Cref{thm:dense-freiman,cor:model-distance} show that if the missing-incidence density is $\mu=o_d(1)$, then deleting $O_d(\sqrt\mu)$ from each side produces one exact lifted-linear model, and this loss controls the natural edit distance in both directions.  \Cref{prop:linear-model-normal-form} gives the complete model list: a rank-$r$ affine system of sphere parameters, $0\le r\le d$, with common base
\[
 S_Q(c_0,r_0)\cap(z+W^{\perp_Q}).
\]
Under codimension-two nonconcentration, \cref{cor:dense-coaxal} forces $r\le1$, leaving a single coaxal pencil.  The rectangular-hole construction of \cref{prop:dense-sqrt-sharp} shows that the square-root stability exponent is best possible.  Thus every dense-endpoint near-extremizer is quantitatively close to an explicit geometric model, with the correct order of deletion loss.

At the lower fixed-radius deviation scale, the universal conclusion is necessarily different.  Once the diagonal is dominated, a positive discrepancy forces common sphere sections at the adaptive scale $K^2q^{d-1}/|\mathcal S|$.  Complete pencils attain this scale.  Under an explicit lower condition, the positive surplus can be resolved, up to a prescribed relative error, into disjoint section-supported pieces; the log-free refinement records additional incidence density on one extracted section.  These are sharp extraction and decomposition theorems, not a classification of every constant-$K$ system.  At the sharp mixed-radius design norm, \cref{thm:mixed-profile} additionally gives a complete quantitative classification of the sphere-intersection and point-degree profiles for both signs.

The reason an individual-template edit classification cannot extend to the whole fixed-scale class is not merely the existence of a few exceptional constructions.  \Cref{thm:hypergraph-universality} fixes one nonzero-radius sphere family whose lifts lie on one rational normal curve and realizes every admissible $d$-uniform hypergraph on the point side.  Dense subhypergraphs already give constant-$K$ moderate-range systems.  By \cref{cor:no-freiman-list}, this class has metric entropy $\exp(\Omega_d(q^d))$: no finite, polynomial-size, or subexponential list of individual point-set templates can approximate every such system within $o(q^d)$ point edits.  A stronger container theorem is not excluded, but its model class would have to permit arbitrary dense subhypergraphs and would therefore record substantially less of the incidence pattern.

The special obstruction systems identify rigid subfamilies inside this universal landscape.  The fixed-radius construction of \cref{thm:fixed-block} is an exact centered direct sum: all cross-block centered incidences vanish and the diagonal blocks carry the entire surplus.  In $d\ge4$ an unbounded number of such blocks fit not only inside the moderate range but also inside the family-only subrange $|\mathcal S|\ge4q/K^2$; no hyperplane or sphere-pair section is macroscopic, and one center-difference direction may still carry asymptotically all coincidence energy.  In $d=3$, \cref{cor:fixed-no-proportional} places every fixed number $t\ge3$ of exact blocks in the same family-only range once $C_1>t/\sqrt2$.

The full rational-normal-curve design is opposite in character.  Its rich-pair graph is complete, but the lifted parameters are in affine general position up to the maximum possible order.  Every pair section carries only $O_d(q^{-2})$ of the surplus, while \cref{thm:rnc-block-indecomp} rules out every $(\tau,\varepsilon)$-centered block decomposition whenever $\varepsilon\le c_d\tau^2$ and $q\ge C_d/\tau$.  The fixed-radius twisted-cubic system of \cref{thm:fixed-tc} shows that, under the point-sensitive hypotheses of \cref{cor:moderate}, even one prescribed nonzero radius does not restore an atom-versus-block alternative in dimension three.  The full diffuse examples do not satisfy $|\mathcal S|\ge4q/K^2$.  Thus a narrower dense-range question remains: in dimension three, does every non-atomic system in that family-only subrange admit an approximate centered block decomposition, or can a diffuse indecomposable branch persist there?

The point--hyperplane and pinned results demonstrate that the coincidence principle is not confined to sphere geometry.  Further applications should be possible whenever object pairs have controlled common neighborhoods and those neighborhoods possess a useful geometric classification.  Negative near-equality at the sharp mixed-radius norm is classified by \cref{thm:mixed-profile}.  Exact complement duality is recorded in \cref{prop:deficit-complement}, but \cref{rem:deficit-scope} shows that its pair-section threshold is below $1$ throughout the balanced regime and becomes structurally informative only when the omitted family is an $O(q^{-1})$ fraction of the full sphere system.  What remains open is an intrinsic inverse theory for deficits inside one prescribed fixed-radius class; that problem is tied to near-extremal Kloosterman and Sali\'e directions rather than to positive coincidence energy.

\appendix
\section{Finite verification and reproducibility}\label{app:verification}

The proofs above are symbolic and do not depend on computation.  To guard against normalization, sign, and finite-field implementation errors, the ancillary source package contains exact-arithmetic scripts with fixed random seeds.  The suite checks the identities and inequalities summarized in \cref{tab:verification}.  Prime fields are implemented by integer arithmetic modulo $p$; the extension-field tests use explicit polynomial-basis arithmetic.  The scripts and a complete transcript are included in the \texttt{anc/} directory of the source package.

\begin{table}[H]
\centering
\footnotesize
\caption{Finite verification suite.  ``Exhaustive'' means every object or assignment in the displayed small instance was enumerated; randomized tests use fixed seeds.}
\label{tab:verification}
\begin{tabularx}{\textwidth}{@{}P{0.31\textwidth}P{0.25\textwidth}Y@{}}
\toprule
Component & Fields and dimensions & Check \\
\midrule
Centered sphere and hyperplane designs & $\mathbb F_3^2,\mathbb F_3^3,\mathbb F_5^2$ & Exhaustive Gram identities, row degrees, and complement-sensitive discrepancy normalization. \\
Radical hyperplanes and coaxal completion & $\mathbb F_3^2,\mathbb F_5^3$ & Exhaustive or randomized equality of common sections along lifted parameter lines. \\
Dense-endpoint classification & Corrupted rank-one models in $\mathbb F_{17}^3$; $75$ random nondegenerate forms over $\mathbb F_7$ in $d=2,3,4$; split rectangular holes in $\mathbb F_{17}^4$ & The good-sphere deletion algorithm, affine-span rank, common base, model-distance inequalities, and square-root sharpness. \\
Scale-adaptive extraction & Small prime-field incidence systems & $912$ randomized systems with positive $\mathcal E_0$, checking the threshold, retained energy, and witness counts. \\
Fixed-radius centered blocks & $\mathbb F_5^4$, $\mathbb F_9^{3,4}$ & Exact block incidence table, constant degrees, common sections, surplus, and family-only inequalities; in $\mathbb F_5^4$, $(|P|,|\mathcal S|,I,\sigma)=(60,15,420,240)$. \\
Rational-normal-curve designs & $d=3,4,5,6$ over $\mathbb F_7,\mathbb F_{11}$; also $\mathbb F_9,\mathbb F_{25}$ & Root-evaluation incidence, pair-section counts, lift independence, and normalized $K$. \\
Hypergraph universality and entropy & $d=3$ over $\mathbb F_5,\mathbb F_7$ & Exact realization of arbitrary subhypergraphs, the formulas $I=d|\mathcal G|$ and $\sigma=(d-1)|\mathcal G|$, and exhaustive Hamming-ball counts in the smallest host. \\
Centered-block indecomposability & Root designs over $\mathbb F_7,\mathbb F_{11}$ & $45{,}000$ fixed-seed randomized matched partitions, checking the theorem's cross-discrepancy lower bound. \\
Fixed-radius twisted cubic & Several prime fields and $\mathbb F_{25}$ & Cubic-root chart, nonzero common radius, complete rich-pair graph, and pair-section counts. \\
Sharp mixed profiles and deficit duality & $\mathbb F_3^2$ and a dense $d=3$ block instance & Positive and negative equality profiles, complement identity, the omitted-family formulas, and the nontriviality diagnostic \eqref{eq:deficit-nontriviality}. \\
\bottomrule
\end{tabularx}
\end{table}

These checks are reproducibility aids rather than evidence substituted for proof.  In particular, the asymptotic obstruction theorems and the profile-stability theorem are established algebraically in the text.

\end{document}